\documentclass[a4paper, 11pt]{amsart}

\usepackage{amsthm, amssymb, wasysym}
\usepackage[leqno]{amsmath}
\usepackage{caption}
\usepackage{subcaption}
\usepackage{hyperref}
\usepackage{booktabs}
\usepackage{relsize}

\usepackage{tikz}
\usepackage{tikz-3dplot}

\usepackage{tikz-cd}
\usepackage{tikz}
\usetikzlibrary{decorations, calc,
	3d, matrix,decorations.pathreplacing,calc,decorations.pathmorphing}
\usetikzlibrary{patterns}

\numberwithin{equation}{section}
\allowdisplaybreaks[1]

\makeatletter
\newcommand{\leqnomode}{\tagsleft@true\let\veqno\@@leqno}
\newcommand{\reqnomode}{\tagsleft@false\let\veqno\@@eqno}
\makeatother

\newcommand{\defi}[1]{{\textbf{#1}}}
\newcommand{\C}{{\mathbb{C}}}
\newcommand{\aaa}[3]{{{a}^{{(#1)}}_{{#2},{#3}}}}

\newcommand{\GW}{\mathrm{GW}}
\newcommand{\Ham}{\mathrm{Ham}}

\definecolor{cof}{RGB}{219,144,71}
\definecolor{pur}{RGB}{186,146,162}
\definecolor{greeo}{RGB}{91,173,69}
\definecolor{greet}{RGB}{52,111,72}

\newtheorem{theorem}{Theorem}[section]
\newtheorem{lemma}[theorem]{Lemma}
\newtheorem{proposition}[theorem]{Proposition}

\newtheorem{Question}[theorem]{Question}

\theoremstyle{definition}
\newtheorem{example}[theorem]{Example}
\newtheorem{definition}[theorem]{Definition}
\newtheorem{remark}[theorem]{Remark}

\begin{document}

\author{Taekgyu Hwang}
\address{Department of Mathematical Sciences
	\\ KAIST \\ 291 Daehak-ro Yuseong-gu \\ Daejeon \\ Republic of Korea 34141}
\email{hwangtaekkyu@gmail.com}
	
\author{Eunjeong Lee}
\address{Department of Mathematical Sciences
	\\ KAIST \\ 291 Daehak-ro Yuseong-gu \\ Daejeon \\  Republic of Korea 34141}
\email{EunjeongLee@kaist.ac.kr}

\author{Dong Youp Suh}
\address{Department of Mathematical Sciences
	\\ KAIST \\ 291 Daehak-ro Yuseong-gu \\ Daejeon \\  Republic of Korea 34141}
\email{dysuh@kaist.ac.kr}

\thanks{This research was supported by Basic Science Research Program through the National Research Foundation of Korea(NRF) funded by the Ministry of Science and ICT(No. 2016R1A2B4010823).}



\title{The Gromov width of generalized Bott manifolds}


\begin{abstract}
By Delzant's theorem, closed symplectic toric manifolds are classified by the images of moment maps. In the case of a generalized Bott manifold, this image is a polytope~$P$ combinatorially equivalent to the product of simplices. We compute the Gromov width of generalized Bott manifolds in terms of the defining inequalities of~$P$. 
\end{abstract}


\date{\today}
\maketitle
\setcounter{tocdepth}{2} \tableofcontents

\section{Introduction}

The Gromov nonsqueezing theorem~\cite{Gr85} asserts that the ball~$B^{2n}(r)$ of radius~$r$ in the Euclidean space~$\mathbb{R}^{2n}$ can be symplectically embedded into the product $B^2(R) \times \mathbb{R}^{2n-2}$ if and only if $r \leq R$. This motivates the following definition of the \defi{Gromov width} of a symplectic manifold $(M, \omega)$:
\begin{equation}
\begin{split}
&w_G(M^{2n}, \omega) \\
& \quad := \sup \{ \pi r^2 \mid B^{2n}(r) \text{ can be symplectically embedded into $M^{2n}$}\}.
\end{split}
\end{equation}
It is difficult to compute the Gromov width in general. To compute this invariant, we should find an upper bound and a lower bound separately  and then check they are equal. We introduce some known methods of estimation below.

	To authors' knowledge, there are only two practical ways to find upper bounds that are available in any dimension. One is to use volumes given by symplectic forms: $(\pi r^2)^n / {n!} = \mathrm{Vol} \, B^{2n}(r) \leq \mathrm{Vol} \, M$. The other way, introduced by Gromov, is to use $J$-holomorphic curves. Given a point $p \in M$, the existence of a certain $J$-holomorphic curve passing through~$p$ obstructs an embedding of a large ball whose image of the origin is~$p$. Such existence is guaranteed when the Gromov--Witten invariant with a point insertion does not vanish, see Theorem~\ref{thm:Gromov}.
	
	Finding lower bounds is basically finding embeddings. There are methods to do this without explicit construction. One of them is to use action-angle coordinates. While the inequality defining the ball in the standard coordinates is quadratic, it becomes linear in action variables in the new coordinates. Hence it becomes easier to check whether the embedding exists. This method is most effective in the case of toric manifolds, since then we have global action-angle coordinates given by a moment map, see Proposition~\ref{prop:embedding}.
	
By the classification result of Delzant~\cite{Del88}, the symplectic structure of a toric manifold is completely determined by its moment polytope. Hence, in principle, all information on the symplectic structure, including the Gromov width, should be recovered from the polytope. But it seems to be out of reach for the moment to find a general formula for the Gromov width, even if we assume toric manifolds are Fano, see Example~\ref{ex:Fano_bundle}. Instead we focus on 
generalized Bott manifolds, the special case of toric manifolds,
and compute the Gromov width in terms of the defining inequalities of the moment polytope.

To state our main theorem we introduce some notation. In the case of generalized Bott manifolds, the moment polytope~$P$ is combinatorially equivalent to a product of simplices. We label facets of~$P$ using two indices; the lower index~$\ell$ indicates the choice of a simplex and the upper index~$k$ indicates the choice of a facet in each simplex. Let $m$ be the number of simplices. For each $1 \leq \ell \leq m$, we let $n_{\ell}$ denote the dimension of the simplex. Then the polytope~$P$ is written as
\begin{equation}\label{eq:P}
	P = \{x \in \mathbb{R}^n \mid \langle x, u_{\ell}^k \rangle \leq \lambda_{\ell}^k \quad \text{ for } \ell = 1, \dots ,m \text{ and } k=0, \dots, n_{\ell}\},
\end{equation}
for some primitive vectors $u_{\ell}^k \in \mathbb{Z}^n$ and real numbers $\lambda_{\ell}^k \in \mathbb{R}$. To simplify the notation, we set
\begin{equation}\label{eq:u(ell)}
	u(\ell):= \sum_{k=0}^{n_{\ell}} u_{\ell}^k \quad \text{and} \quad \lambda(\ell):= \sum_{k=0}^{n_{\ell}} \lambda_{\ell}^k.
\end{equation}
We remark that there is at least one~$\ell$ satisfying $u(\ell) = 0$ (see~\eqref{eq_u_k_ell}), and $\lambda(\ell) > 0$ for such~$\ell$. Our main theorem is the following.
\begin{theorem}\label{thm:main}
	Let $M$ be a generalized Bott manifold whose symplectic form is given by the polytope~$P$ in~\eqref{eq:P}. Then the Gromov width is given by
	\[
		w_G(M) = \min_{1 \leq \ell \leq m} \{ \lambda(\ell) \mid u(\ell) = 0 \}.
	\]
\end{theorem}

Combining results by Caviedes Castro~\cite{CC16} and Fang--Littelmann--Pabiniak~\cite{FLP17}, the Gromov width of partial flag manifolds were recently computed. On the other hand, we do not know the general formula for the Gromov width of toric manifolds. There are some related partial results. Many of them consider the case when toric manifolds are Fano, probably because it is easier to compute the Gromov--Witten invariants. Under the Fano assumption, the upper bound given by Theorem~1.2 in~\cite{Lu06} is particularly useful. See Theorem~\ref{thm:Lu} for the statement. This estimate can be confirmed to be sharp in many cases. Example~\ref{ex:Fano_bundle} might be a possible candidate where such estimate is not optimal, but we do not know any method of finding a sharper bound. Moreover, even without the Fano condition, we do not know any example of a toric manifold for which this estimate is not an upper bound. We summarize these questions in Section~\ref{sec:Examples}.

Theorem~\ref{thm:main} is proved by combining Proposition~\ref{prop_lower_bounds_Gromov_width} and Proposition~\ref{prop:upper_bound}. We essentially use the combinatorics of the moment polytope~$P$ in the proof. Lemma~\ref{lemma_checking_inequalities_for_P}, which is a simple observation by the fact that $P$ is combinatorially equivalent to the product of simplices, plays a crucial role in the proof of Proposition~\ref{prop_lower_bounds_Gromov_width}. This is also used in the proof of Proposition~\ref{prop:upper_bound}. Since we do not assume our generalized Bott manifolds to be Fano, we cannot use Theorem~\ref{thm:Lu} to get an upper bound. Instead, we use the results on the Seidel representation by McDuff and Tolman~\cite{MT06}. The Seidel morphism is a group homomorphism
\[
	S \colon \pi_1(\Ham(M,\omega)) \to QH^0(M; \Lambda)^{\times}
\]
where the group structure on the target is given by the (small) quantum product. Since the quantum product is defined using the genus zero Gromov--Witten invariants, we obtain some information on the Gromov--Witten invariants by studying this homomorphism. Using this information together with the combinatorics of~$P$, we obtain the same upper bound as in Theorem~\ref{thm:Lu} without Fano assumption.

\section{Generalized Bott manifolds}\label{sec:GBM}
Generalized Bott manifolds are projective toric manifolds 
whose topology were studied in \cite{CMS10-quasitoric} and  
\cite{CMS-Trnasaction}. We first introduce generalized Bott manifolds
as toric manifolds, and then consider symplectic structures on them.

\begin{definition}[{\cite[Definition 6.1]{CMS10-quasitoric}}]
	A \defi{generalized Bott tower} $\{B_j \mid j=0,\dots,m\}$ of
	height $m$ (or an {$m$-stage generalized Bott tower}) is a sequence,
	\[
	\begin{tikzcd}
	B_m \arrow[r, "\pi_m"]
	& B_{m-1} \arrow[r, "\pi_{m-1}"]
	& \cdots \arrow[r, "\pi_3"]
	& B_2 \arrow[r, "\pi_2"]
	& B_1 \arrow[r, "\pi_1"]
	& B_0 = \{\text{a point}\},
	\end{tikzcd}
	\]
	of manifolds $B_j = \mathbb{P}(E_j^{(1)} \oplus \cdots \oplus E_j^{(n_j)} \oplus
	\underline{\C})$ where $E_j^{(k)}$ is a holomorphic line bundle over 
	$B_{j-1}$ for $k=1, \dots, n_j$ and $\underline{\C}$ is the trivial line bundle over~$B_{j-1}$,
	and $\mathbb{P}(\cdot)$ stands for the induced projective bundle.
	We call $B_j$ the \defi{$j$-stage generalized Bott manifold}
	of the generalized Bott tower.
\end{definition}
\begin{example}
	Every complex projective space $\mathbb{P}^{n+1}$ is a $1$-stage generalized Bott manifold.
	Also the product of projective spaces $\mathbb{P}^{n_1+1} \times 
	\cdots \times \mathbb{P}^{n_m+1}$ is a generalized Bott manifold. 
	When $n_j=1$ for all $j=1,\dots,m$, the generalized
	Bott tower is called the Bott tower, which was first defined in~\cite{GrKa94}.
\end{example}

It is known from \cite[Exercise II.7.9]{Ha77} that 
the Picard group of the $j$-stage generalized Bott manifold $B_{j}$ is the free abelian group of rank~$j$ for
$j=1,\dots,m$. 
Hence for each holomorphic line bundle $E_j^{(k)}$
over $B_{j-1}$, we are given $j-1$ many integers $\aaa{k}{j}{1},\dots,\aaa{k}{j}{j-1}$
associated to the line bundle~$E_j^{(k)}$.

A projective bundle of the sum of holomorphic line bundles 
over a toric manifold is known to be a toric manifold (see~\cite[Proposition 7.3.3]{CLS11}).
Hence an $m$-stage generalized Bott manifold $B_m$ is
a toric manifold. The fan structure of~$B_m$ is described as follows.
Put ${n:= \sum_{j=1}^m n_j}$.
The fan~$\Sigma$ of an $m$-stage generalized Bott manifold~$B_m$ is a smooth complete fan in $\mathbb{R}^{n}$ with~${n+m}$ many rays. Let $\{ e^1_1,\dots,e^{n_1}_1,\dots,e^1_m,\dots,e^{n_m}_m\}$ be the standard basis of $\mathbb{R}^n$. 
The ray generators 
$u^0_1,\dots,u^{n_1}_1,
\dots,u^0_m, \dots,u^{n_m}_m$ in $\mathbb{Z}^n$
are given by
\begin{align}
	\begin{split}\label{eq_u_k_ell}
		u^k_{\ell} &= e^k_{\ell} \quad \text{ for }k = 1,\dots,n_{\ell},\\
		u^0_{\ell} &= (\underbrace{\mathbf 0}_{n_1}, \dots, \underbrace{\mathbf 0}_{n_{\ell-1}}, 
		\underbrace{\mathbf{-1}}_{n_{\ell}}, \underbrace{\mathbf{a}_{\ell+1,\ell}}_{n_{\ell+1}},
		\dots, \underbrace{\mathbf{a}_{m,\ell}}_{n_m}),
	\end{split}
\end{align}
for $\ell=1,\dots,m$, where
\begin{equation}
	\mathbf{a}_{j,\ell}:= (\aaa{1}{j}{\ell},\dots,\aaa{n_j}{j}{\ell}) \in \mathbb{Z}^{n_j}.
\end{equation}
Moreover there is a one-to-one correspondence between the set of maximal cones in $\Sigma$ 
and the set $S$ of sequences:
\begin{equation}\label{eq:vertex}
	S:= \{ (s_1,\dots,s_m) \mid s_{\ell} \in \{0,1,\dots,n_{\ell}\} \quad \text{for } \ell =1,\dots,m\}.
\end{equation}
For a given sequence $(s_1,\dots,s_m) \in S$,
we have a maximal cone generated by ray generators
\[
\{u^k_{\ell} \mid k \in \{0,1,\dots,n_{\ell}\} \setminus \{s_{\ell}\} \text{ and } \ell=1,\dots,m\}.
\]
Since we have $|S| = (n_1+1)\cdots (n_m+1)$, there are
$(n_1+1)\cdots (n_m+1)$ many maximal cones.
\begin{example}
Consider the case~$m=2$ with $n_1 = 1$ and $n_2 = 1$. Suppose four ray generators are given by
\[
u^0_1 = (-1,-1),\quad u^1_1 = (1,0),\quad u^0_2 = (0,-1), \quad u^1_2 = (0,1).
\]
We have four maximal cones
\[
\text{Cone}(u^0_1, u^0_2), \quad
\text{Cone}(u^1_1, u^0_2), \quad
\text{Cone}(u^0_1, u^1_2), \quad
\text{Cone}(u^1_1, u^1_2).
\]
The fan $\Sigma$ defines a generalized Bott manifold,
which is the Hirzebruch surface $\mathcal{H}_1 = \mathbb{P}(\mathcal{O}(1)
\oplus \underline{\C})$. 
See Figure~\ref{figure:Hirzebruch_fan}. 
\end{example}

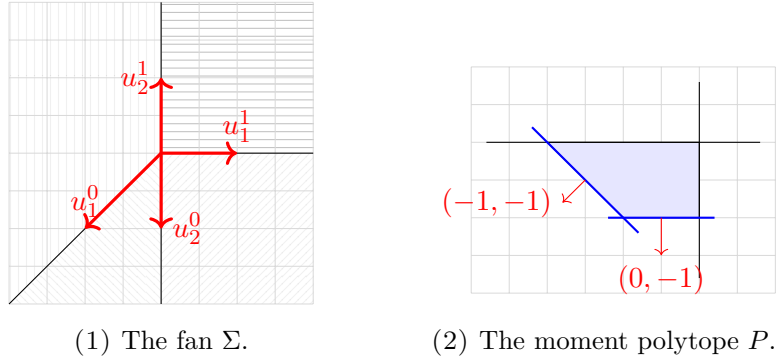
\begin{figure}
	\begin{subfigure}[t]{0.45\textwidth}
		\begin{center}
			\begin{tikzpicture}
			
			\fill[pattern = horizontal lines, nearly transparent] (0,0) -- (0,2) -- (2,2) -- (2,0) -- cycle;
			\fill[pattern = vertical lines, nearly transparent] (0,0) -- (0,2) -- (-2,2) -- (-2,-2) -- cycle;
			\fill[fill=purple!10!white, pattern = north west lines, nearly transparent] (0,0) -- (-2,-2) -- (0,-2) -- cycle;
			\fill[fill=green!10!white, pattern = north east lines, nearly transparent] (0,0) -- (2,0) -- (2,-2) -- (0,-2) -- cycle;
			
			\draw[step = 0.5, gray!30!white, very thin] (-2,-2) grid (2,2);
			
			\draw (0,0) -- (0,2);
			\draw (0,0) -- (2,0) ;
			\draw (0,0) -- (-2,-2);
			\draw (0,0) -- (0,-2);
			\draw[very thick,->,red] (0,0) -- (1,0) node[ anchor =south] {$u^1_1$};
			\draw[very thick,->,red] (0,0) -- (0,1) node[anchor =east] {$u^1_2$};
			\draw[very thick, ->,red] (0,0) -- (0,-1) node[anchor = west] {$u^0_2$};
			\draw[very thick, ->,red] (0,0) -- (-1,-1) node[anchor = south] {$u^0_1$};
	
			\end{tikzpicture}
		\end{center}
		\caption{The fan $\Sigma$.}
		\label{figure:Hirzebruch_fan}
	\end{subfigure}
	\begin{subfigure}[t]{0.45\textwidth}
		\begin{center}
			\begin{tikzpicture}
			\draw[step = 0.5, gray!30!white, very thin] (-3,-2) grid (1,1);
			\filldraw[fill=blue!10!white] (-2,0) -- (-1,-1) -- (0,-1) -- (0,0) -- cycle;
			
			\draw (-2.8,0) -- (0.8,0);
			\draw (0,-1.8) -- (0,0.8);
			
			\draw[thick, blue] (-2.2,0.2) -- (-0.8,-1.2);
			\draw[thick, blue] (-1.2,-1) -- (0.2,-1);
			
			\draw[red,->] (-1.5,-0.5) -- (-1.8,-0.8) 
			node[anchor = east] {$(-1,-1)$};
			\draw[red,->] (-0.5,-1) -- (-0.5,-1.5) node[anchor = north] {$(0,-1)$};
			\end{tikzpicture}
		\end{center}
		\caption{The moment polytope $P$.}
		\label{figure_example_polytope_P}
	\end{subfigure}
	\caption{The fan $\Sigma$ and the moment polytope $P$ of
		the Hirzebruch surface $\mathcal{H}_1$.}
\end{figure}

Let $B_m$ be a generalized Bott manifold and $\Sigma$ be the fan of $B_m$.
Let $\lambda_1,\dots,\lambda_m$ be real numbers.
Then the fan $\Sigma$ and these real numbers define the following
polytope $P$.
\begin{equation}\label{equation_P_equalities}
P = \{  x \in \mathbb{R}^n_{\leq 0} \mid 
\langle  x, u_{\ell}^0 \rangle \leq \lambda_{\ell} \quad
\text{ for } \ell = 1,\dots,m \}.
\end{equation}
We remark that the total space of a projective bundle over 
a projective variety is a projective variety (see~\cite[Section 7.0]{CLS11}).
Hence a generalized Bott manifold $B_m$ is a projective variety since 
$B_m$ is obtained by an iterated sequence of projectivizations.
Then there exists a polytope whose normal fan is $\Sigma$ 
by Theorem~7.2.10 in \cite{CLS11}.
Therefore there exist 
integers $\lambda_1,\dots,\lambda_m$ such that 
the polytope in~\eqref{equation_P_equalities}
is combinatorially equivalent to the product of simplices $\prod_{j=1}^m
\Delta^{n_j}$. 

We describe a symplectic structure on a generalized Bott manifold $B_m$ 
using Delzant's Theorem which says that a Delzant polytope determines
a symplectic toric manifold (see \cite{Del88}, \cite[Theorem 2.1.2]{CS01Symplectic}).
Here we call a polytope $P \subset \mathbb{R}^n$ \defi{Delzant} if it
satisfies:
\begin{itemize}
	\item there are $n$ edges meeting at each vertex;
	\item for each vertex $p$, each edge is of the form $p + t u_i$,
	$t \geq 0$ where $u_i \in \mathbb{Z}^n$ and vectors
	$u_1,\dots,u_n$ form 
	a $\mathbb{Z}$-basis of $\mathbb{Z}^n$.
\end{itemize}
Suppose that $\lambda_1,\dots,\lambda_m$ are real numbers such that
the polytope $P$ in \eqref{equation_P_equalities} is combinatorially
equivalent to 
the product of simplices $\prod_{j=1}^m \Delta^{n_j}$.
Then one can easily see that the polytope $P$ is a Delzant polytope, and moreover the normal fan of $P$ is the fan $\Sigma$ of $B_m$. 
Hence
we obtain a symplectic toric manifold $(B_m, \omega)$ from this polytope~$P$.
\begin{example}
Let $m=2$ with $n_1= 1$ and $n_2 = 1$. 
Let $u^0_1 = (-1,-1)$, $u^0_2 = (0,-1)$, $\lambda_1 = 2$ and
$\lambda_2 = 1$.
Then the polytope
$P \subset \mathbb{R}^2$ in \eqref{equation_P_equalities} is defined by
\[
P = \{ x \in \mathbb{R}^2_{\leq 0} 
\mid \langle x, (-1,-1) \rangle \leq 2, ~~
\langle x, (0,-1) \rangle \leq 1\}.
\]
As one can see in Figure~\ref{figure_example_polytope_P} the polytope $P$ is combinatorially
equivalent to the product $\Delta^1 \times \Delta^1$ of simplices. Hence we obtain a symplectic toric
manifold $(\mathcal{H}_1, \omega)$ from the polytope $P$.
\end{example}

We describe the combinatorics of the polytope $P$
more precisely for later uses.
Let $\{v_j^0, v_j^1,\dots,v_j^{n_j}\}$ 
be the set of vertices of the $n_j$-simplex $\Delta^{n_j}$. 
A facet of the polytope $P$ is a product of a facet of
one of $\Delta^{n_j}$'s and the remaining simplices. Hence
the set of facets of the polytope $P$ is 
\[
\mathcal{F}(P) = \bigcup_{\ell=1}^m \{ F_{\ell}^k 
\mid k = 0,1,\dots,n_{\ell}\},
\]
where 
\[
F_{\ell}^k = \Delta^{n_1} \times \cdots \times \Delta^{n_{\ell-1}} 
\times f_{\ell}^k \times \Delta^{n_{\ell+1}} \times \cdots \times \Delta^{n_m}.
\]
Here $f_{\ell}^k \subset \Delta^{n_{\ell}}$ is the facet which is opposite to the
vertex $v_{\ell}^k \in \Delta^{n_{\ell}}$ for $k=0,1,\dots,n_{\ell}$. 

Let $I$ be a subset of $[m]:= \{1,\dots,m\}$. Consider
the face 
\begin{equation}\label{eq_face_FI}
F_I := \bigcap_{\ell \in [m] \setminus I}
\bigcap_{k \in [n_{\ell}]} F^k_{\ell}.
\end{equation}
Since the facet $F^k_{\ell} \subset P$ can be described by
$
F^k_{\ell} =
\{x \in \mathbb{R}^{n} \mid 
\langle x, u^k_{\ell} \rangle = 0
\} \cap P$ for $k = 1,\dots,n_{\ell}$,
we have that
\[
\begin{split}
F_I = \{
x = (\mathbf x_1,\dots,\mathbf x_m) \in \mathbb{R}^{n_1}_{\leq 0} \oplus 
\cdots \oplus \mathbb{R}^{n_m}_{\leq 0}
\mid 
& \langle x, u_{\ell}^0 \rangle \leq \lambda_{\ell} \text{ if } \ell \in I\text{, and } \\
&\quad 
\mathbf x_{\ell} = 0 \text{ if } \ell \in [m] \setminus I
\}.
\end{split}
\]
The following lemma directly comes from the above observation:
\begin{lemma}\label{lemma_checking_inequalities_for_P}
Suppose that a point $x = (\mathbf x_1, \dots,\mathbf x_m) \in \mathbb{R}^{n_1}_{\leq 0} \oplus 
\cdots \oplus \mathbb{R}^{n_m}_{\leq 0}$ satisfies that
\begin{align*}
	\begin{cases}
		\langle x, u^0_{\ell} \rangle \leq \lambda_{\ell} 
		\quad &\text{ if } \ell \in I, \\
		\; \mathbf x_{\ell} = 0 \quad &\text{ if } \ell \in [m] \setminus I
	\end{cases}
\end{align*}
for some subset $I \subset [m]$. Then $x$ is contained in $P$.
\end{lemma}

\section{Lower bounds}\label{sec:lower_bounds}

In this section we find lower bounds for the Gromov width of generalized Bott manifolds. By Proposition~\ref{prop:embedding}, the existence of a symplectic embedding of a ball into the symplectic toric manifold is guaranteed by a certain polytope embedded in the moment polytope~$P$. Although it is enough to find a simplex in~$P$ in our case, we state the proposition in a more general form for the arguments in Section~\ref{sec:Examples}. An explicit construction of the embedded simplex in~$P$ is given in Proposition~\ref{prop_lower_bounds_Gromov_width}.

We begin by recalling the Arnold--Liouville theorem (see for example \cite[Theorem~III.3.3]{Aud04}). The theorem asserts that, given a completely integrable system on a symplectic manifold $(M^{2n}, \omega)$, a compact connected regular level set~$L$ is a Lagrangian torus. Moreover, $L$ has a neighborhood~$U$ with coordinates
\[
	(\mathbf{a}, \pmb{\alpha}) \in \mathbb{R}^n \times (\mathbb{R}/\mathbb{Z})^n
\]
such that $L$ is given by a level set of~$\mathbf{a}$, and $\omega = \sum da_i \wedge d\alpha_i$ on~$U$. Such coordinates are called \defi{action-angle coordinates}. When our symplectic manifold $(M, \omega)$ is toric, the integrable system given by the moment map $\mu \colon M \to P \subset \mathbb{R}^n$ provides us global coordinates in the sense of Duistermaat~\cite[Theorem~2.2]{Dui80}. Precisely, we have the following symplectomorphism (see \cite[Remark~IV.4.19]{Aud04})
\begin{equation}\label{eq:action-angle}
	\left(\mu^{-1}(\mathrm{Int}\, P), \omega\right) \cong \left(\mathrm{Int}\, P \times T^n, \sum da_i \wedge d\alpha_i\right).
\end{equation}

The method using action-angle coordinates to find a symplectic embedding was used by many authors, including \cite{Traynor95, Sch05, Lu06, LMS13, MP17}. To explain this method in more detail, we define some terminologies. The word ``distorted'' in the following definition is taken from~\cite[Section~4.2]{LMS13}.
\begin{definition}
A \defi{distorted cross-polytope of length~$\rho$ of dimension~$n$}, denoted by $\Diamond^n(\rho)$, is the convex hull of line segments $L_1,\dots,L_n$ in~$\mathbb{R}^n$ satisfying the following conditions:
\begin{itemize}
	\item $L_1 \cap \cdots \cap L_n = \{p\}$ for some point $p \in \mathbb{R}^n$ (we say $\Diamond^n(\rho)$ is centered at~$p$);
	\item the primitive vectors parallel to the line segments~$L_i$ form a basis for $\mathbb{Z}^n$; and
	\item each $L_i$ has length $\ell(L_i) = \rho$.

\end{itemize}
\end{definition}
Here the \defi{length}~$\ell(L)$ of the line segment $L$ with rational slope is defined as follows. After translation and multiplication by some scalar~$c>0$, we may assume that the line segment~$c \cdot L$ has integral endpoints. Then the length is given by
\begin{equation}\label{eq:length}
	\ell(L) := \frac{\#(c \cdot L \cap \mathbb{Z}^n)-1}{c}.
\end{equation}
For instance, if $L$ connects $(0,0)$ and $(3/2,3)$ then $2 \cdot L$ connects integral points, and the length of $L$ is~$3/2$, see Figure~\ref{figure_length_of_a_line}. We remark that the simplex~$\Delta(\rho)$ with edge length~$\rho$ is a distorted cross-polytope centered at a vertex of~$\Delta(\rho)$.
\begin{figure}
\begin{tikzpicture}[scale=0.5]
\draw[step = 1, gray!30!white, very thin] (-1,-1) grid (7,7);

\draw (-0.8,0) -- (6.8,0);
\draw (0,-0.8) -- (0,6.8);

\draw[thick, blue, dashed] (0,0) -- (3,6) node[anchor = east] {$2\cdot L$};
\draw[very thick, green!50!black] (0,0) -- (1.5, 3) node[anchor = east] {$L$};

\foreach \x/\xtext in {1,...,6}
	\draw(\x,1pt) -- (\x,-1pt) node[anchor = north] {\small{$\xtext$}};
\foreach \y/\ytext in {1,...,6}
	\draw(1pt,\y) -- (-1pt,\y) node[anchor = east] {\small{$\ytext$}};

\fill[orange!30!red] (0,0) circle (4pt) ;
\fill[orange!30!red] (1,2) circle (4pt) ;
\fill[orange!30!red] (2,4) circle (4pt) ;
\fill[orange!30!red] (3,6) circle (4pt) ;

\end{tikzpicture}
\caption{The length of a line segment}
\label{figure_length_of_a_line}
\end{figure}
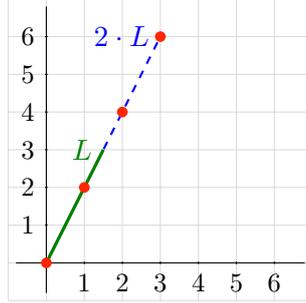

\begin{definition}
For $0 < s \leq 1$, let $f_s \colon \mathbb{R} \to \mathbb{R}_{\geq 0}$ be a piecewise linear function defined by
\begin{align}
	f_s(a) :=
	\begin{cases}
	a / s  &\text{ if } a \geq 0, \\
	a / (s-1)&\text{ if }a < 0 \text{ and } s < 1,\\
	0& \text{ if }a<0 \text{ and } s =1.
	\end{cases}
\end{align}
For any $\rho>0$ and $0 < s_i \leq 1$, $i=1, \dots, n$, define
\begin{equation}
	\Diamond_{s_1, \dots, s_n}(\rho):= \left\{(a_1, \dots, a_n) \in \mathbb{R}^n \; \middle| \; \sum_{i=1}^n f_{s_i}(a_i) \leq \rho \right\}.
\end{equation}
\end{definition}
It follows from the definition that $\Diamond_{s_1, \dots, s_n}(\rho)$ is the convex hull of line segments connecting two points $\rho (s_i - 1) e_i$ and $\rho s_i e_i$ for $i=1, \dots, n$, where $e_i$ denotes the $i$th standard basis vector of $\mathbb{R}^n$. Hence, it is a distorted cross-polytope of length~$\rho$ centered at the origin.

\begin{proposition}[\cite{LMS13} Section~4.2, \cite{MP17} Proposition~5]\label{prop:embedding}
	Let $(M, \omega)$ be a symplectic toric manifold of dimension~$2n$ with the moment polytope~$P \subset \mathbb{R}^n$. If the polytope~$P$ contains a distorted cross-polytope~$\Diamond^n(\rho)$, then the Gromov width of $(M, \omega)$ is at least~$\rho$.
\end{proposition}
\begin{proof}
	We could not find an explicitly written proof in the literature, so the proof is included for completeness.
	
	After suitable change of basis by $\mathrm{GL}(n, \mathbb{Z})$ and translation, we may assume that $\Diamond^n(\rho) = \Diamond_{s_1, \dots, s_n}(\rho)$ for some numbers $0 < s_1, \dots, s_n \leq 1$. Fix a small number $\epsilon > 0$. For $0 < s \leq 1$, we can find an area preserving embedding (see Lemma~3.1.5 in~\cite{Sch05})
\begin{equation}
	\sigma_s \colon B^2(\sqrt{\rho/\pi} - \epsilon) \to \mathbb{R} \times (0,1) : (x,y) \mapsto (a,\alpha)
\end{equation}
such that for all $r > \epsilon$,
\begin{equation*}
	x^2 + y^2 < (r - \epsilon)^2 \implies (s-1)\pi r^2 < a(\sigma_s(x,y)) < s \pi r^2.
\end{equation*}
This map is illustrated in Figure~\ref{fig:sigma}. By the definition of~$f_s$, it follows that
\begin{equation}
	x^2 + y^2 < (r - \epsilon)^2 \implies f_s(a(\sigma_s(x,y))) < \pi r^2.
\end{equation}
Since each~$\sigma_{s_i}$ preserves the symplectic structure, the product map
\[
	\sigma := \sigma_{s_1} \times \dots \times \sigma_{s_n} \colon B^2(\sqrt{\rho/\pi} - \epsilon)^n \to \mathbb{R}^n \times (0,1)^n
\]
is a symplectic embedding. Considering the image under~$\sigma$ of the point $(\mathbf{x},\mathbf{y}) = (x_1, \dots, x_n, y_1, \dots, y_n) \in B^{2n}(\sqrt{\rho/\pi} - \epsilon)$, we have
\begin{align*}
	\quad \sum_{i=1}^n f_{s_i}(a_i(\sigma(\mathbf{x}, \mathbf{y}))) 
	= \sum_{i=1}^n f_{s_i}(a(\sigma_{s_i}(x_i, y_i))) < \sum_{i=1}^n \pi (x_i^2 + y_i^2) < \rho.
\end{align*}
This inequality shows that 
\[
	\sigma(B^{2n}(\sqrt{\rho/\pi} - \epsilon)) \subset \Diamond^n(\rho) \times (0,1)^n \subset \Diamond^n(\rho) \times T^n.
\]	
Using the symplectomorphism in~\eqref{eq:action-angle}, we obtain a symplectic embedding of $B^{2n}(\sqrt{\rho/\pi} - \epsilon)$ into~$M$.
\end{proof}

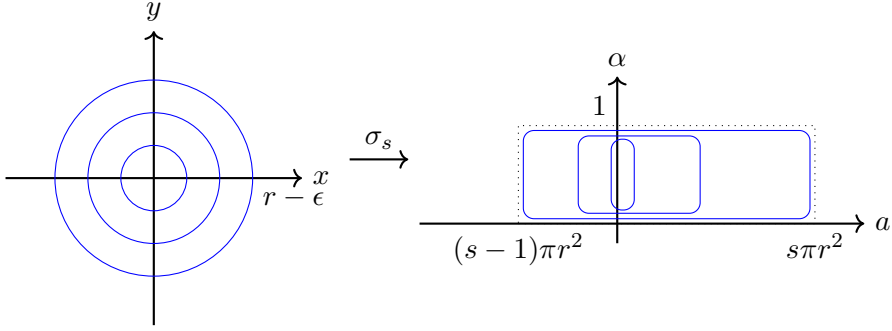
\begin{figure}
\begin{tikzpicture}
{
	\def\k{3};
	\node (a) at (0,0)
	{
		\begin{tikzpicture}[scale=1.3]
		{
			\draw[->, thick] (-1.5,0) -- (1.5,0) node [right] {$x$};
			\draw[->, thick] (0,-1.5) -- (0,1.5) node [above] {$y$};
			\foreach \j in {1,...,\k} \draw[color = blue] (0,0) circle (\j/\k);
			\node [below right] at (1,0) {$r - \epsilon$};
		}
		\end{tikzpicture}
	};
	\node (b) at (a.east)[anchor=east, xshift = 210]
	{
		\begin{tikzpicture}[scale=1.3]
		{
			\draw[->, thick] (-2,0) -- (2.5,0) node [right] {$a$};
			\draw[->, thick]  (0,-.2) -- (0,1.5) node [above] {$\alpha$};
			\foreach \j in {1,...,\k} \draw[rounded corners, color = blue]
			({-(\j/\k)^2 + 0.05}, {0.1*(1- (\j/\k)^2) + 0.05}) -- ({2* (\j/\k)^2 - 0.05}, {0.1*(1- (\j/\k)^2) + 0.05})
			-- ({2* (\j/\k)^2 - 0.05}, {1 - 0.1*(1- (\j/\k)^2) - 0.05}) -- ({-(\j/\k)^2 + 0.05}, {1 - 0.1*(1- (\j/\k)^2) - 0.05}) -- cycle;
			\node [below] at (2,0) {$s \pi r^2$};
			\node [below] at (-1,0) {$(s-1) \pi r^2$};
			\node [above left] at (0,1) {$1$};
			\draw[rectangle, dotted] (-1,0) -- (2,0) -- (2,1) -- (-1,1) -- cycle;
		}
		\end{tikzpicture}
	};
	\draw[->, thick] (a) -- node [above] {$\sigma_s$} (b);
}
\end{tikzpicture}
\caption{The symplectic map $\sigma_s$}
\label{fig:sigma}
\end{figure}

From now on we focus on generalized Bott manifolds. Let $(B_m, \omega)$ be the generalized Bott manifold with the moment polytope
\begin{equation*}
	P = \{  x \in \mathbb{R}^n_{\leq 0} \mid \langle  x, u_{\ell}^0 \rangle \leq \lambda_{\ell} \text{ for } \ell = 1,\dots,m \}
\end{equation*}
as in~\eqref{equation_P_equalities}. No inequalities should be redundant for defining~$P$, so the real numbers~$\lambda_{\ell}$ are all positive. Since one can change the order of simplices if necessary, there exists an integer $1 \leq d \leq m$ such that
\begin{equation}
	u(\ell) = u_{\ell}^0 + u_{\ell}^1 + \cdots + u_{\ell}^{n_{\ell}} = 0
	\iff \ell \geq d.
\end{equation}
Define a real number $\lambda$ as
\begin{equation}
	\lambda := \min \{ \lambda_d,\dots, \lambda_m \}.
\end{equation}
Note that $\lambda$ coincides with the number given in Theorem~\ref{thm:main}.

\begin{proposition}\label{prop_lower_bounds_Gromov_width}
Let $P$ and $\lambda$ be as above. Then there exists an $n$-simplex~$\Delta(\lambda)$ of length~$\lambda$ embedded in $P \subset \mathbb{R}^n$.
\end{proposition}
\begin{proof}
It is enough to find $n$~points 
$v^1_1,\dots, v^{n_1}_1,\dots, v^1_m, \dots, v^{n_m}_m$ in $P$ such that the set
\[
\{v^1_1/\lambda,\dots, v^{n_1}_1/\lambda,
\dots, v^1_m/\lambda,\dots,v^{n_m}_m/\lambda\}
\]
 forms a $\mathbb{Z}$-basis. 
Let $\{e^1_1,\dots,e^{n_1}_1,\dots,e^1_m,\dots,e^{n_m}_m \}$ 
be the standard basis of~$\mathbb{R}^n$. 
Define points $v^k_{\ell} \in \mathbb{R}^n$ inductively in descending order of~$\ell$ as follows:
\begin{equation}\label{equation_def_of_vjk}
v^k_{\ell}
:= \begin{cases}
- \lambda e^k_{\ell} & \text{if } \ell = m \text{ or } \mathbf{a}_{j,\ell} \in \mathbb{Z}^{n_{j}}_{\leq 0}
\text{ for all } j=\ell+1,\dots,m, \\
-\lambda e^k_{\ell}  + v^{k'}_{\ell'}& \text{otherwise}.
\end{cases}
\end{equation}
In the second case, the pair $(\ell',k')$ is defined as
\begin{itemize}
	\item 
	$\ell' := \max \{ j > \ell \mid \mathbf{a}_{j,\ell} \notin 
	\mathbb{Z}^{n_j}_{\leq 0}\}$,
	\item $k' := \max \{ i \leq n_{\ell'} \mid \aaa{i}{\ell'}{\ell} > 0\}$.
\end{itemize}
The condition on $\ell'$ implies that 
we have $\mathbf a_{j,\ell} \in \mathbb{Z}^{n_j}_{\leq 0}$ for all
$j > \ell'$. The second condition implies that 
$\aaa{k'+1}{\ell'}{\ell},\dots,
\aaa{n_{\ell'}}{\ell'}{\ell} \leq 0$. Hence the last positive entry of the vector $u^0_{\ell}$ is $\aaa{k'}{\ell'}{\ell}$.

Note that the set of points $\{v^k_{\ell}/\lambda\}$ 
forms a $\mathbb{Z}$-basis. Hence it remains to show
the following claim to prove the proposition:

\vspace{1em}
\noindent \textsf{\textbf{Claim.}}
The point $v_{\ell}^k$ is contained in the polytope~$P$ for all $\ell = 1, \dots, m$ and $k = 1, \dots, n_{\ell}$.
\vspace{1em}

\noindent We prove this claim by induction on the number of nonzero entries of the vector~$v_{\ell}^k$, say~$N$.
\vspace{1em}

\noindent\textsf{\textbf{Case 1: \fbox{$v^k_{\ell} = -\lambda e^k_{\ell}$.}}}
This is the first step of the induction when $N=1$. Since the point $v^k_{\ell}
=: (\mathbf v_1,\dots,\mathbf v_m)$ 
is contained in $\mathbb{R}^n_{\leq 0}$
and $\mathbf v_j = \mathbf 0 \in \mathbb{R}^{n_j}_{\leq 0}$ for all $j \neq \ell$, by Lemma \ref{lemma_checking_inequalities_for_P} it is enough to show the inequality
\begin{equation}\label{equation_lambda_and_lambda_ell}
\lambda = \langle v^k_{\ell}, u_{\ell}^0 \rangle \leq \lambda_{\ell}.
\end{equation}

If $\ell \geq d$, then by the definition of $\lambda$, we have $\lambda \leq \lambda_{\ell}$. 
To consider the case when $\ell < d$, we define the sequence $P_{\ell}$ of the pairs of indices
\begin{equation}\label{eq_def_P_ell}
P_{\ell} = ((\ell,k),(i_1,p_1),\dots,
(i_M, p_M))
\end{equation}
inductively as follows:
\begin{itemize}
	\item $i_{v} := 
	\max \{ j > i_{v-1} \mid \mathbf{a}_{j,i_{v-1}} \neq 
	\mathbf 0\}$ with $i_0 = \ell$,
	\item $p_v := \max \{ q \leq n_{i_v} \mid
	\aaa{q}{i_v}{i_{v-1}} \neq 0\}$.
\end{itemize}
The sequence stops when there is no such~$i_v$. We have $i_M \geq d$; otherwise it contradicts the definition of $d$. Note from the definition of the sequence~$P_{\ell}$ that $\ell = i_0 < i_1 < \cdots < i_M \leq m$. The condition on $(i_v, p_v)$ implies that the last nonzero entry of the vector~$u^0_{i_{v-1}}$ is $\aaa{p_v}{i_v}{i_{v-1}}$.

We call a sequence $I = ((j_0,q_0),(j_1,q_1),\dots,(j_{s},q_{s}))$ of pairs of indices with $s \geq 1$ and $j_0<j_1 <  \cdots < j_{s}$
\defi{admissible} if
\begin{equation}
	\begin{cases}
		\aaa{q_1}{j_1}{j_0}  &< 0, \\
		\aaa{q_t}{j_t}{j_{t-1}} &> 0 \quad \text{for } t = 2,\dots,s, \\
		\aaa{q_{t_1}}{j_{t_1}}{j_{t_2}} &=0 \quad \text{for } t_1 - t_2  \geq 2.
	\end{cases}
\end{equation}
We remark that the case when $s=1$ is allowed, so there might be no positive~$\aaa{q_t}{j_t}{j_{t-1}}$. The proof of the following lemma will be provided after the proof of the proposition.
\begin{lemma}\label{lemma_admissible}
Let $I = ((j_0,q_0),(j_1,q_1),\dots,(j_{s},q_{s}))$ be an
admissible sequence. Then $\lambda_{j_{0}} \geq \lambda_{j_s}$.
\end{lemma}
Note from the definition~\eqref{eq_def_P_ell} that $\aaa{p_v}{i_v}{i_{v-1}} \neq 0$ for all~$v$ and $\aaa{p_{v_1}}{i_{v_1}}{i_{v_2}} = 0$ for $v_1 - v_2 \geq 2$. Since $\aaa{p_1}{i_1}{\ell} < 0$, we can always split the sequence~$P_{\ell}$ into admissible pieces as
\[
\begin{split}
& \left((\ell,k), (i_1,p_1), \dots, (i_{b_1},p_{b_1}) \right), ~~
\left((i_{b_1},p_{b_1}),\dots, (i_{b_2},p_{b_2})\right),\dots,\\
&\qquad \left((i_{b_x},p_{b_x}),\dots,(i_M,p_M) \right).
\end{split}
\]
Then by Lemma~\ref{lemma_admissible}, we have inequalities
\[
	\lambda_{\ell} \geq \lambda_{i_{b_1}} \geq
	\cdots \geq \lambda_{i_{b_x}} \geq \lambda_{i_M}.
\]
On the other hand, we have $\lambda_{i_M} \geq \lambda$ since $i_M \geq d$. The inequality~\eqref{equation_lambda_and_lambda_ell} follows by combining these inequalities. This proves the claim for the case $N=1$.

\vspace{1em}
\noindent\textbf{\textsf{Case 2: \fbox{$v^k_{\ell} = -\lambda e^k_{\ell}  + v^{k'}_{\ell'}$.}}}
Note that the number of nonzero entries of~$v_{\ell}^k$ is greater by one than that of~$v_{\ell'}^{k'}$. By the induction hypothesis, we have $\langle v_{\ell'}^{k'}, u_t^0 \rangle \leq \lambda_t$ for $t = 1, \dots, m$. When $t > \ell$, since $\langle e_{\ell}^{k}, u_t^0 \rangle = 0$ we have the inequality
\begin{equation}\label{eq:case2}
	\langle v^k_{\ell}, u_t^0 \rangle \leq \lambda_t.
\end{equation}
Now consider the case when $t = \ell$. There are two possibilities
\begin{equation*}
	v_{\ell'}^{k'} =
	\begin{cases}
	- \lambda e_{\ell'}^{k'} &\text{or,}\\
	-\lambda e_{\ell'}^{k'}  + v_{\ell''}^{k''} \quad &\text{for some $\ell''$ and $k''$.}
	\end{cases}
\end{equation*}
Since $\langle v_{\ell''}^{k''}, u_{\ell}^0 \rangle \leq \lambda_{\ell}$ by the induction hypothesis, in either case we have
\[
	\langle v^k_{\ell}, u_{\ell}^0 \rangle
	\leq \langle - \lambda e^k_{\ell} - \lambda e^{k'}_{\ell'}, u_{\ell}^0 \rangle + \lambda_{\ell}
	= \lambda (1 -  \aaa{k'}{\ell'}{\ell}) + \lambda_{\ell}
	\leq \lambda_{\ell},
\]
where the last inequality comes from the fact that $\aaa{k'}{\ell'}{\ell} \geq 1$ by the definition of~$\ell'$ and~$k'$.

We have proved that the inequality~\eqref{eq:case2} holds when $t \geq \ell$. Now the inequality~\eqref{eq:case2} holds for any $t=1, \dots, m$ by Lemma~\ref{lemma_checking_inequalities_for_P}, which completes the induction argument.
\end{proof}

\begin{proof}[Proof of Lemma \ref{lemma_admissible}]
Let $I = ((j_0,q_0),(j_1,q_1),\dots,(j_{s},q_{s}))$ be an admissible sequence. Consider the set $J := \{j_1,\dots,j_{s}\}$ and let $F_J$ be the face of $P$ defined as in \eqref{eq_face_FI}. First we claim that
\[
	v := -\lambda_{j_{s}}(e^{q_1}_{j_1}+\cdots+e^{q_s}_{j_s}) \in F_J.
\]
By Lemma~\ref{lemma_checking_inequalities_for_P}, it is enough to show
\begin{equation}\label{eq_point_v_inequalities}
	\langle v, u^0_{\ell} \rangle \leq \lambda_{\ell} \quad \text{for } \ell \in J.
\end{equation}
Recall the following expression
\[
	u^0_{\ell} = (\underbrace{\mathbf{0},\dots, \mathbf{0}}_{\ell-1},
	\mathbf{-1}, \mathbf{a}_{\ell+1, \ell}, \dots, \mathbf{a}_{m, \ell}).
\]
By the definition of the admissible sequence, we have
\begin{align*}
	\langle e_{j_t}^{q_t}, u_{\ell}^0 \rangle
	= \begin{cases}
		\aaa{q_t}{j_t}{\ell} \geq 1 &\text{for } \ell = j_{t -1},\\
		-1 &\text{for } \ell = j_t, \\
		0 &\text{otherwise.}
	\end{cases}
\end{align*}
Then $\langle v, u_{j_s}^0 \rangle = \lambda_{j_s}$ and 
\[
\langle v, u_{j_t}^0 \rangle = -\lambda_{j_s} (-1 + \aaa{q_{t+1}}{j_{t+1}}{j_t}) \leq 0 < \lambda_{j_s}
\] 
for $t<s$, which proves the inequality~\eqref{eq_point_v_inequalities}
so that $v \in F_J \subset P$.

Since $v \in F_J \subset P$, we have the inequality
$\langle v, u^{0}_{j_0} \rangle \leq \lambda_{j_0}$.
Computing the left hand side we obtain
\[
\langle v, u^0_{j_0} \rangle
= -\lambda_{j_s} \aaa{q_1}{j_1}{j_0} \geq \lambda_{j_s}
\]
since the value $\aaa{q_1}{j_1}{j_0}$ is negative.
Hence we have the inequality
$\lambda_{j_s} \leq \lambda_{j_0}$.
\end{proof}

\begin{example}\label{example_embedded_simplex_B5}
Consider the generalized Bott manifold~$B_5$ with $n_j = 1$ for $j \leq 4$ and $n_5 = 2$. The ray vectors are given by the following matrix
\[
	[u^0_1 \ u^0_2 \ u^0_3 \ u^0_4 \ u^0_5] =
	\left[
	\begin{array}{c|c|c|c|c}
		-1 & 0 & 0 & 0 & 0\\ \hline
		-3 & -1 & 0 & 0 & 0\\ \hline
		0 & 1 & -1 & 0 & 0\\ \hline
		0 & -2 & -4 & -1 & 0\\ \hline
		0 & -1 & 0 & -2 & -1\\
		0 & 0 & 0 & 2 & -1
	\end{array}
	\right].
\]
Then $d = 5$, so put $\lambda = \lambda_5$. Let $\{ e_1^1, e_2^1, e_3^1, e_4^1, e_5^1, e_5^2 \}$ be the standard basis of~$\mathbb{R}^6$. Since there are no positive entries in the fifth, the third and the first column, we take $v_5^1 = -\lambda e_5^1$, $v_5^2 = -\lambda e_5^2$, $ v_3^1 = -\lambda e_3^1$ and $v_1^1 = - \lambda e_1^1$. In the fourth column, the last positive entry~$2$ is located at the 2nd place of the 5th block. We take $v_4^1 = -\lambda e_4^1 + v_5^2 = -\lambda(e_4^1 + e_5^2)$. In the second column the last positive entry~$1$ is located at the 1st pace of the 3rd block, hence $v_2^1 = -\lambda e_2^1 + v_3^1 = -\lambda(e_2^1 + e_3^1)$. We have the following six points in $\mathbb{R}^6$:
\begin{equation}
\begin{split}
v_{1}^1 = - \lambda e_1^1, \quad 
v_{2}^1 = -\lambda(e_2^1 + e_3^1),\quad
v_{3}^1 = - \lambda e_3^1, \\
v_{4}^1 = - \lambda (e_4^1 + e_5^2), \quad 
v_{5}^1 = - \lambda e_5^1, \quad
v_{5}^2 = - \lambda e_5^2.
\end{split}
\end{equation}
We can compare the values~$\lambda_{\ell}$ using Lemma~\ref{lemma_admissible}. For the $\ell$th column with non-positive entries with $\ell < d$, we have sequences
\[
	P_1 = ((1,1), (2,1), (5,1)) \quad \text{and} \quad P_3 = ((3,1), (4,1), (5,2)).
\]
The sequence~$P_3$ is admissible but $P_1$ is not. Then we split $P_1$ into two admissible sequences $((1,1), (2,1))$ and $((2,1), (5,1))$. By Lemma~\ref{lemma_admissible},
\begin{equation}\label{eq:lambda}
	\lambda_1 \geq \lambda_2 \geq \lambda_5 \quad \text{and} \quad \lambda_3 \geq \lambda_5.
\end{equation}
We demonstrate the last inequality following the proof of Lemma~\ref{lemma_admissible}. In this case $J = \{4,5\}$. Using Lemma~\ref{lemma_checking_inequalities_for_P}, the vector $v:= -\lambda_5 (e_4^1 + e_5^2)$ is contained in $F_J \subset P$ since
\begin{align}
	\begin{split}\label{eq:v}
		& \langle v, u_4^0 \rangle = \langle (0,0,0, -\lambda_5, 0, -\lambda_5), (0,0,0,-1,-2,2) \rangle = -\lambda_5 < \lambda_4, \\
		& \langle v, u_5^0 \rangle = \langle (0,0,0, -\lambda_5, 0, -\lambda_5), (0,0,0,0,-1,-1) \rangle = \lambda_5.
	\end{split}
\end{align}
Solving the inequality $\langle v, u_3^0 \rangle \leq \lambda_3$, we obtain $\lambda_5 \leq 4\lambda_5 \leq \lambda_3$ which proves the last inequality in~\eqref{eq:lambda}.

To check whether the points~$v_{\ell}^k$ are contained in $P$, we use Lemma~\ref{lemma_checking_inequalities_for_P} again. For $\ell = 1,3,5$, it is enough to check $\langle v_{\ell}^k, u_{\ell}^0 \rangle \leq \lambda_{\ell}$. The left hand side is equal to~$\lambda = \lambda_5$, so the inequality follows from~\eqref{eq:lambda}. The case when $\ell = 4$ was done in~\eqref{eq:v}. Considering the case $\ell = 2$, we check
\begin{align*}
& \langle v_{2}^1, u_2^0 \rangle = 
\langle (0, -\lambda, -\lambda, 0, 0, 0),(0, -1, 1, -2, -1, 0) \rangle = 
0 < \lambda_2, \\
&\langle v_{2}^1, u_3^0 \rangle
= \langle (0, -\lambda, -\lambda, 0, 0, 0), (0, 0, -1, -4, 0, 0) \rangle
= \lambda_5 \leq  \lambda_3.
\end{align*}
Therefore we have six points $v_1^1, v_2^1, v_3^1, v_4^1, v_5^1, v_5^2$ which, together with the origin, form a $6$-simplex of length~$\lambda$ embedded in~$P$. 
\end{example}

\section{Upper bounds}\label{section:upper_bound}
In this section, we find an upper bound for the Gromov width using the Gromov--Witten invariants. In the case of generalized Bott manifolds, Lemma~\ref{lem:divisible} provides us more information on the terms given in Theorem~\ref{thm:MT}. The upper bound is given in Proposition~\ref{prop:upper_bound}, whose proof uses this information together with Theorem~\ref{thm:Gromov}.

First we briefly review the definition of the Gromov--Witten invariants and explain how they can be used to estimate the Gromov width. We only consider the genus zero invariants with primary insertions because these are sufficient for our purposes. We refer the reader to~\cite{MS12} for details for the case when the symplectic manifold is semi-positive. The definition of the Gromov--Witten invariants for general symplectic manifolds requires the construction of the virtual fundamental class and can be found in the recent articles~\cite{Cas16, McD14, Par16, TF17}. In these definitions, the virtual fundamental class may be considered as an element in the rational \v Cech homology group of the moduli space. See \cite{ES52} for the definition and properties of \v Cech homology groups. See also \cite[ Section~1]{IP17}, \cite[Section~4.3]{TF17}  and  Remark~8.2.4. in~\cite{MW18} with explanations below it.

Let $(M, \omega)$ be a closed symplectic manifold of dimension~$2n$.  Given an $\omega$-tame almost complex structure~$J$ and $A \in H_2(M; \mathbb{Z})$, consider the moduli space~ $\overline{\mathcal{M}}_{0,k}^M(A,J)$ of $J$-holomorphic stable maps of genus zero to~$M$ in the class~$A$ with $k$ marked points. Using the notation $c_1(A) := \langle c_1(TM, J), A \rangle$, this space carries a virtual fundamental class in the rational \v Cech homology group
\[
	[\overline{\mathcal{M}}]^{vir} \in \check H_{2n + 2c_1(A)+ 2k - 6} \, (\overline{\mathcal{M}}_{0,k}^M(A,J); \mathbb{Q})
\]
which is independent of the choice of~$J$. Let $ev \colon \overline{\mathcal{M}}_{0,k}^M(A,J) \to M^k$ be the evaluation map and $\pi \colon \overline{\mathcal{M}}_{0,k}^M(A,J) \to \overline{\mathcal{M}}_{0,k}$ be the forgetful map, whose target is a smooth manifold of dimension $2k-6$. For $\beta \in H_*(\overline{\mathcal{M}}_{0,k}; \mathbb{Q})$ and $\alpha_i \in H^*(M; \mathbb{Q})$ for $i=1, \dots, k$, the \defi{Gromov--Witten invariant} is defined to be the rational number
\begin{equation}
		\GW_{A,k}^M(\alpha_1, \dots, \alpha_k; \beta):= \langle (ev \times \pi)_* [\mathcal{\overline{M}}]^{vir}, \alpha_1 \times \dots \times \alpha_k \times PD(\beta) \rangle \in \mathbb{Q}.
\end{equation}

The following theorem follows from the definition of the Gromov--Witten invariants and the idea Gromov used to prove his nonsqueezing theorem.
\begin{theorem}[Gromov]\label{thm:Gromov}
Let $(M, \omega)$ be a closed symplectic manifold. If $\GW_{A,k}^M([pt], \alpha_2, \dots, \alpha_k; \beta) \neq 0$ for some $A \neq 0$, $k \geq 3$, $\alpha_i \in H^*(M; \mathbb{Q})$ and $\beta \in H_*(\overline{\mathcal{M}}_{0,k}; \mathbb{Q})$, then $w_G(M) \leq \omega(A)$.
\end{theorem}
\begin{proof}
	Let $N \subset M^{k-1} \times \overline{\mathcal{M}}_{0,k}$ be a submanifold representing (a nonzero multiple of) $\alpha_2 \times \dots \times \alpha_k \times PD(\beta)$. By the assumption, the image of $\overline{\mathcal{M}}_{0,k}^M(A,J)$ under $ev \times \pi$ intersects~$p \times N$ for any $p \in M$ and $J$. In other words, for any $\omega$-tame almost complex structure~$J$ and any point $p \in M$, there exists a $J$-holomorphic stable map $u \colon \Sigma \to M$ in the homology class~$A$ such that $p \in \mathrm{Im} \, u$. Here the domain~$\Sigma$ may be considered as a connected union of spheres. Since $A \neq 0$, there exists a non-constant component $u_i \colon S^2 \to M$ in the class $A_i$ whose image contains~$p$. Now the proof follows from the argument of Gromov~\cite{Gr85}, which we sketch here for the reader's convenience.
	
	Let $\iota \colon B^{2n}(r) \hookrightarrow M$ be a symplectic embedding with $p:= \iota(0)$. Let $J_0$ denote the standard complex structure of~$\mathbb{C}^n$. Then $\iota_*J_0$ together with~$\omega$ determines a Riemannian metric~$g_0$ on the image. For any small $\epsilon > 0$, using the partition of unity, we can construct a Riemannian metric~$g$ on~$M$ such that $g=g_0$ on~$\iota \left(B^{2n}(r - \epsilon)\right)$. We take an almost complex structure~$J$ so that $g(\cdot,\cdot) = \omega(\cdot, J \cdot)$. Now consider
	\[
		C:= \iota^{-1}(\mathrm{Im}\, u_i) \cap B^{2n}(r - \epsilon).
	\]
	Since $C$ is the image of a $J_0$-holomorphic curve $\iota^{-1} \circ u_i$ for the standard complex structure~$J_0$, it is a minimal surface in $B^{2n}(r - \epsilon)$ passing through the origin. By the monotonicity property of minimal surfaces, we have
	\[
		\pi (r - \epsilon)^2 \leq \mathrm{Area} \, C \leq \omega(A_i) \leq \omega(A).
	\]
	Since $\epsilon$ can be chosen arbitrarily, we obtain the desired inequality.
\end{proof}

\begin{remark}
	It seems reasonable to expect that Theorem~\ref{thm:Gromov} holds for any definition of the Gromov--Witten invariants in the existing literature but we could not check this.
\end{remark}

To find such Gromov--Witten invariants, it is useful to use the Seidel representation on the (small) quantum cohomology ring. We begin with recalling the definition of the quantum cohomology ring. Details can be found in Chapter~11 of~\cite{MS12}.

Consider the Novikov ring
\begin{equation}\label{eq:Novikov}
	\Lambda:= \Lambda^{\mathrm{univ}}[q,q^{-1}]
\end{equation}
where $q$ is a variable of degree~$2$ and
	\[
	\Lambda^{\mathrm{univ}}:= \left\{\sum_{i \in \mathbb{N}}
	r_i t^{\kappa_i} \;\Big|\; r_i \in \mathbb{Q}, \; \kappa_i \in \mathbb{R}, \; \lim_{i \to \infty} \kappa_i = \infty \right\}
	\]
with $\deg t =0$. The \defi{quantum cohomology ring} with coefficients in~$\Lambda$ is an abelian group
\begin{equation}
	QH^*(M; \Lambda):= H^*(M; \mathbb{Q}) \otimes_{\mathbb{Q}} \Lambda
\end{equation}
together with the product~$\ast$ defined as follows. Let $a \in H^i(M; \mathbb{Q})$ and $b \in H^j(M; \mathbb{Q})$. Then
\begin{equation}
	a*b := \sum_{A \in H_2(M; \mathbb{Z})} (a*b)_A \otimes q^{c_1(A)} t^{\omega(A)}
\end{equation}
where $(a*b)_A \in H^{i+j-2c_1(A)}(M; \mathbb{Q})$ is defined uniquely by the condition
\begin{equation}
	\int_M (a*b)_A \cup c = \GW_{A,3}^M (a,b,c;[pt])
\end{equation}
for all $c \in H^*(M; \mathbb{Q})$. This product extends linearly on~$\Lambda$ and is called the \defi{quantum product}. The associativity of the product follows from the splitting axiom of the Gromov--Witten invariants. Moreover, iterated products can be computed as follows: for $a_i \in H^{d_i}(M; \mathbb{Q})$, their product is equal to
\begin{equation}
	a_1 * \dots * a_k = \sum_{A \in H_2(M; \mathbb{Z})} (a_1 * \dots * a_k)_A \otimes q^{c_1(A)} t^{\omega(A)}
\end{equation}
where $(a_1 * \dots * a_k)_A \in H^{(\sum_{i=1}^k d_i) -2c_1(A)}(M; \mathbb{Q})$ is defined by the relation
\begin{equation}\label{eq:quantum_iterated_product}
	\int_M (a_1 * \dots * a_k)_A \cup c = \GW_{A,k+1}^M(a_1, \dots, a_k, c; [pt])
\end{equation}
for all $c \in H^*(M; \mathbb{Q})$. In the special case when $A=0$, we remark that
\begin{equation}\label{eq:quantum_product_A=0}
	(a_1 * \dots * a_k)_0 = a_1 \cup \dots \cup a_k.
\end{equation}

The Seidel representation was first introduced in~\cite{Sei97}. We will use the version defined in~\cite{LMP99}, where the authors define a group homomorphism, called the \defi{Seidel morphism},
\begin{equation}
	S \colon \pi_1(\Ham(M,\omega)) \to QH^0(M; \Lambda)^{\times}
\end{equation}
from the fundamental group of the Hamiltonian diffeomorphism group to the abelian group of the degree~$0$ units in the quantum cohomology ring. Since the quantum product is defined using the Gromov--Witten invariants, the Seidel morphism can be used to detect some nonvanishing Gromov--Witten invariants. For example, McDuff used this homomorphism in~\cite{McD09} to show that any Hamiltonian $S^1$-manifold is uniruled.\footnote{A symplectic manifold is called uniruled if there is a nonzero genus zero Gromov--Witten invariant with a point insertion.} When a circle acts on~$M$ in Hamiltonian way, we have a nice representative of the loop in $\Ham(M, \omega)$. In~\cite{MT06}, McDuff and Tolman developed methods to compute the image of such loops under~$S$. Since toric symplectic manifolds are our main interest, we explain their results when $(M, \omega)$ is toric.

Let $(M, \omega)$ be a closed symplectic manifold of dimension~$2n$ with an effective Hamiltonian $T^n$-action. Let $\mu \colon M \to \mathbb{R}^n$ be its moment map whose image is a convex simple polytope~$P$. Then for some primitive vectors $\eta_i \in \mathbb{Z}^n$ and real numbers $\kappa_i \in \mathbb{R}$ for $i=1, \dots, N$, we have
\begin{equation}
	P = \{ x \in \mathbb{R}^n \mid \langle x, \eta_i \rangle \leq \kappa_i \quad\text{ for } i=1, \dots, N\}.
\end{equation}
Each vector~$\eta_i$ determines a Hamiltonian sub-circle action on~$M$, hence represents a loop in~$\Ham(M, \omega)$. We say a $J$-holomorphic stable map is \defi{$\eta_i$-invariant} if the image is invariant under the corresponding circle action. The map defined by $\mu_i:= \langle \mu, \eta_i \rangle \colon M \to \mathbb{R}$ is a moment map for this action. Let $D_i$ denote the toric divisor corresponding to~$\eta_i$, which is the maximum fixed component with respect to the moment map~$\mu_i$. The moment map~$\mu$ or the polytope~$P$ is called \defi{normalized} if $\int_M \mu_i \, \omega^n =0$ for all~$i$. This is equivalent to the condition that the center of mass of~$P$ is located at the origin.
\begin{theorem}[Theorem~1.10 and Lemma~3.10 in~\cite{MT06}]\label{thm:MT}
Let $\mu \colon M \to P$ be the normalized moment map. Then the following hold.
\begin{enumerate}
	\item The image of~$\eta_i$ under the Seidel morphism is given by
\[
	S(\eta_i) = \sum_{B \in H_2(M; \mathbb{Z})} a_i(B) \otimes q^{-1 + c_1(B)} t^{-\kappa_i + \omega(B)}
\]
for some $a_i(B) \in H^{2-2c_1(B)}(M; \mathbb{Q})$.
	\item $a_i(B) = 0$ when $\omega(B) < 0$ or $\omega(B) = 0$ with $B \neq 0$.
	\item $a_i(0) = [D_i]$. Here $[\cdot]$ denotes the cohomology class represented by a submanifold.
	\item $a_i(B) = 0$ for $B \neq 0$ when $M$ is Fano.
	\item Let $N$ be an $\eta_i$-invariant submanifold of~$M$. Then $\int_M a_i(B) \cup [N] = 0$ if the homology class~$B$ cannot be represented by an $\eta_i$-invariant \mbox{$J$-holomorphic} stable map intersecting both the toric divisor~$D_i$ and the submanifold~$N$. 
\end{enumerate}
\end{theorem}

To make use of the toric action to compute~$a_i(B)$, we consider a slight variation of Lemma~3.10 from~\cite{MT06}.
\begin{lemma}\label{lem:T-inv}
	In the last statement of Theorem~\ref{thm:MT}, we can replace an \mbox{$\eta_i$-invariant} $J$-holomorphic stable map with a $T^n$-invariant one.
\end{lemma}
\begin{proof}
	All the theorems we mention here are from~\cite{MT06}. Since the proof goes almost the same as the original one, we only sketch the proof of Lemma~3.10 and add some change in the toric case.
	
	The key theorem used in the proof of Lemma~3.10 is Proposition~4.10. It asserts that when a symplectic manifold has an $S^1$-action, the Gromov--Witten invariants can be computed by considering only invariant stable maps. The $S^1$-action~$\eta_i$ constructs a Hamiltonian fiber bundle~$P_{\eta_i}$ over~$S^2$ with fiber~$M$. This bundle itself is a symplectic manifold and carries a $T^2$-action: fiberwise $S^1$-action~$\eta_i$ and an additional $S^1$-action rotating the base. The classification of $T^2$-invariant holomorphic sections (more precisely, stable maps in the section class) of~$P_{\eta_i} \to S^2$ is given by Lemma~3.8: each $T^2$-invariant holomorphic section consists of the section~$\sigma_z \colon S^ 2 \to P_{\eta_i}$ with value at some fixed point~$z \in M$ of~$\eta_i$, together with $\eta_i$-invariant stable maps whose image intersect~$\sigma_z(S^2)$ and contained in the fibers~$M_0$ or~$M_{\infty}$ over $\{0,\infty\} \subset S^2$.
	
	The cohomology class~$a_i(B)$ is defined using the Gromov--Witten invariant counting holomorphic sections of~$P_{\eta_i}$. Proposition~3.4 asserts that by choosing an appropriate $S^1 \subset T^2$ action on~$P_{\eta_i}$, the $S^1$-invariant holomorphic sections coincide with the ones invariant under~$T^2$. Now the proof is done by showing that, when $\int_M a_i(B) \cup [N] \neq 0$, the point~$z$ can be chosen in~$D_i$ and the stable map contained in~$M_0$ can be chosen to intersect~$N$.
		
	In the toric case, choose $w = (w_1, \dots, w_n) \in \mathbb{Z}^n$ so that the induced circle action given by the inclusion $t \mapsto (t^{w_1}, \dots, t^{w_n})$ has the same fixed point set as $T^n$-action. For example, we may choose $w$ satisfying $\langle w, v-v' \rangle \neq 0$ for all vertices~$v$ and~$v'$ of the moment polytope~$P$. Now the $w$-invariant stable maps coincide with the $T^n$-invariant ones. All arguments in Lemma~3.10 work in the same way if we replace the $\eta_i$-action with the $w$-action.
\end{proof}

From now on we consider the case when $M$ is a generalized Bott manifold. Recall from Introduction that the moment polytope is written as follows:
\begin{equation}
	P = \{x \in \mathbb{R}^n \mid \langle x, u_{\ell}^k \rangle \leq \lambda_{\ell}^k \quad\text{ for } \ell = 1, \dots m \text{ and } k=0, \dots, n_{\ell}\}.
\end{equation}
Also recall the notations
\begin{equation}
	u(\ell):= \sum_{k=0}^{n_{\ell}} u_{\ell}^k \quad \text{and} \quad \lambda(\ell):= \sum_{k=0}^{n_{\ell}} \lambda_{\ell}^k.
\end{equation}
Real numbers $\lambda_{\ell}^k$ may change if we normalize the polytope~$P$. However, when $u(\ell) = 0$, the sum $\lambda(\ell)$ does not change. Since the expression given in Theorem~\ref{thm:main} is not changed, from now on we assume $P$ is normalized in this section. Note that $\lambda_{\ell}^k > 0$ when $P$ is normalized.

As in  Theorem~\ref{thm:MT}, the image of~$u_{\ell}^k$ by the Seidel morphism is written as
\begin{equation}
	S(u_{\ell}^k) = \sum_{B \in H_2(M; \mathbb{Z})} a_{\ell}^k(B) \otimes q^{-1+c_1(B)} t^{-\lambda_{\ell}^k + \omega(B)}.
\end{equation}
Recall that $F_{\ell}^k$ denotes a facet of~$P$. Let $D_{\ell}^k := \mu^{-1}(F_{\ell}^k)$ denote the toric divisor corresponding to~$F_{\ell}^k$. 
We need the following lemma to obtain an upper bound for the Gromov width.
\begin{lemma}\label{lem:divisible}
	Suppose $u(\ell) = 0$ and $\omega(B) \leq \lambda(\ell)$ for some~$B \in H_2(M; \mathbb{Z})$. Then the cohomology class $a_{\ell}^k(B)$ is divisible by~$[D_{\ell}^k]$ in $H^*(M; \mathbb{Q})$.
\end{lemma}
\begin{proof}
	It is known that the cohomology ring of a toric manifold is generated as a ring by toric divisor classes corresponding to facets of~$P$. Products of distinct toric divisor classes correspond to intersections of facets of~$P$, and they generate the cohomology ring as an abelian group.\footnote{See for example \cite[Section~5]{Ful93}.} So in the following of the proof, given a face~$F$ of~$P$, we let $[F]$ denote the corresponding cohomology class. In particular, $[F_{\ell}^k] = [D_{\ell}^k]$. In the case of generalized Bott manifolds, any face disjoint from~$F_{\ell}^k$ is contained in the face $F(k):= \cap_{i \neq k} F_{\ell}^i$. Hence, the proof is complete if we show $\int_M a_{\ell}^k(B) \cup [F] = 0$ for all faces~$F$ contained in~$F(k)$.
		
	 Suppose not. Then by Theorem~\ref{thm:MT} and Lemma~\ref{lem:T-inv}, the homology class~$B$ is represented by a $T^n$-invariant $J$-holomorphic stable map $u \colon \Sigma \to M$ intersecting both~$D_{\ell}^k$ and~$D(k):= \cap_{i \neq k} D_{\ell}^i$. Since the image of~$u$ is $T^n$-invariant, the moment map image $\mu(u(\Sigma))$ is contained in the $1$-skeleton of~$P$. The image is also connected, so there is an edge~$E$ of~$P$ contained in $\mu(u(\Sigma))$ connecting $F_{\ell}^k$ and~$F(k)$.
	 
	 We compute the intersection number $[E] \cdot [D_p^q]$ as follows. Recall from~\eqref{eq:vertex} that the vertex $E \cap F(k)$ can be identified with the sequence~$(s_1, \dots, s_m)$ with $0 \leq s_j \leq n_j$. Note from the definition of~$F(k)$ that $s_{\ell} = k$. Since the edge~$E$ connects~$F_{\ell}^k$ and~$E \cap F(k)$, there exists~$k' \neq k$ such that
	\begin{equation}\label{eq:edge}
	 	E =  \left(\bigcap_{i \neq k, k'} F_{\ell}^i \right) \cap \left(\bigcap_{j \neq \ell} \bigcap_{i \neq s_j} F_j^i \right).
	\end{equation}
	We claim that
	\begin{equation}\label{eq:intersection}
		[E] \cdot [D_p^q] =
		\begin{cases}
			1,& \text{ if }p=\ell \text{ and}\\
			0,& \text{ otherwise}.
		\end{cases}
	\end{equation}
	From the expression~\eqref{eq:edge}, the cohomology class~$[E]$ has $n_{\ell}-1$ factors of~$[D_{\ell}^i]$ and $n_j$ factors of~$[D_j^i]$ for~$j \neq \ell$. If the factors in~\eqref{eq:intersection} are all distinct, the claim follows since the intersection number comes from the intersection of the corresponding facets. In the case when $[E]$ has $[D_p^q]$~factor, we use linear relations of divisors to remove the multiplicity. Explicitly,
	\begin{equation}\label{eq:linear_relation}
		[D_p^0] = [D_p^i] + \sum_{j=1}^{p-1} \aaa{i}{p}{j} [D_j^0] \quad \text{for } i = 1, \dots, n_p.
	\end{equation}
	Then for any~$p$ and~$q$, we can write as
	\[
		[D_p^q] = [D_p^{s_p}] + \sum_{j=1}^{p-1} c_j [D_j^0], \quad \text{where } c_j = \aaa{s_p}{p}{j} - \aaa{q}{p}{j} \text{ with convention } \aaa{0}{p}{j}=0.
	\]
	By the expression~\eqref{eq:edge} and the fact that $\cap_{i=0}^{n_j} F_j^i$ is empty, the term~$[D_p^{s_p}]$ contributes to the intersection number in~\eqref{eq:intersection} if and only if $p = \ell$, and the contribution is one. In the remaining terms~$c_j[D_j^0]$ with~$j < p$, the case~$j = \ell$ does not appear since the assumption $u(\ell) = 0$ implies that $\aaa{q}{p}{\ell} = 0$ for all~$p$ and~$q$. Whenever the multiplicity occurs with~$[D_j^0]$, that is $s_j \neq 0$, we apply the relation~\eqref{eq:linear_relation} for $p=j$ and $i=s_j$. Repeating this procedure, all the multiplicities are removed without creating factors of the form~$[D_{\ell}^i]$, and the contributions are all zero by~\eqref{eq:edge}.

	Since $[\omega] = \sum \lambda_{\ell}^k [D_{\ell}^k]$ and $c_1(M) = \sum [D_{\ell}^k]$,\footnote{See Equation~(1.6) in \cite{Gui94}, where a different convention on the sign and the normalization was used.} it follows from~\eqref{eq:intersection} that
	\[
		\omega([E]) = \lambda(\ell) \quad \text{and} \quad c_1([E]) = n_{\ell}+1.
	\]
	Since $J$-holomorphic curves have non-negative symplectic areas, we conclude that $B = [E]$ by the assumption $\omega(B) \leq \lambda(\ell)$. But then the degree of~$a_{\ell}^k(B)$ is negative since $\deg a^k_{\ell}(B) + 2(-1 + c_1(B)) =0$, which yields a contradiction.
\end{proof}

\begin{proposition}\label{prop:upper_bound}
	If $u(\ell) = 0$, then $w_G(M) \leq \lambda(\ell)$.
\end{proposition}
\begin{proof}
	Since $u(\ell) = 0$ and $S$ is a group homomorphism,
	\begin{align*}
		1 &= \prod_{k=0}^{n_{\ell}} S(u_{\ell}^k)
		= \prod_{k=0}^{n_{\ell}} \left(\sum_{B_k \in H_2(M; \mathbb{Z})} a_{\ell}^k(B_k) \otimes q^{-1 + c_1(B_k)} t^{-\lambda_{\ell}^k + \omega(B_k)}\right) \\
		&= \sum_{A, B_k \in H_2(M; \mathbb{Z})} \left(a_{\ell}^0(B_0) * \dots * a_{\ell}^{n_{\ell}}(B_{n_{\ell}})\right)_A \\
		&\qquad \qquad \qquad \qquad \otimes q^{-(n_{\ell}+1) + c_1(A + \sum_{k=0}^{n_{\ell}} B_k)} t^{-\lambda(\ell) + \omega(A + \sum_{k=0}^{n_{\ell}}B_k)}.
	\end{align*}
	So there exist $A', B_{k}' \in H_2(M; \mathbb{Z})$ for $k=0, \dots, n_{\ell}$ such that
	\begin{align}\label{eq:upper_bound}
		\begin{cases}
		\left(a_{\ell}^0(B_{0}') * \dots * a_{\ell}^{n_{\ell}}(B_{n_{\ell}}')\right)_{A'} &\neq 0, \\
		c_1\left(A' + \sum_{k=0}^{n_{\ell}} B_k'\right) &= n_{\ell}+1, \\
		\omega \left(A' + \sum_{k=0}^{n_{\ell}} B_k'\right) &= \lambda(\ell).
		\end{cases}
	\end{align}
	We claim that $A' \neq 0$. Recall from the equation~\eqref{eq:quantum_iterated_product} that the homology class~$A'$ is represented by a $J$-holomorphic stable map, so $\omega(A') \geq 0$. We also have $\omega(B_k') \geq 0$ for all~$k$ by Theorem~\ref{thm:MT} and the first equation of~\eqref{eq:upper_bound}. Now $\omega(B_k') \leq \lambda(\ell)$ for all~$k$ by  the third equation of~\eqref{eq:upper_bound}, so the product $a_{\ell}^0(B_0') \cup \dots \cup a_{\ell}^{n_{\ell}}(B_{n_{\ell}}')$ is zero since it is a multiple of $\prod_{k=0}^{n_{\ell}} [D_{\ell}^k] = 0$ by Lemma~\ref{lem:divisible}. We thus have $A' \neq 0$ from~\eqref{eq:quantum_product_A=0}. Since
	\[
		\deg \left(a_{\ell}^0(B_0') * \dots * a_{\ell}^{n_{\ell}}(B_{n_{\ell}}')\right)_{A'}
		= \sum_{k=0}^{n_{\ell}} \left(2 - c_1(B_k')\right) - 2c_1(A') = 0,
	\]
	we have
	\[
	\begin{split}
		&\GW_{A',n_{\ell}+2}^{M} \left(a_{\ell}^0(B_0'), \dots, a_{\ell}^{n_{\ell}}(B_{n_{\ell}}'), [pt]; [pt]\right)  \\
		&\qquad = \int_M \left(a_{\ell}^0(B_0') * \dots * a_{\ell}^{n_{\ell}}(B_{n_{\ell}}')\right)_{A'} \cup [pt] \neq 0.
	\end{split}
	\]
	It follows that $w_G(M) \leq \omega(A') \leq \lambda(\ell)$ by Theorem~\ref{thm:Gromov} and the equation~\eqref{eq:upper_bound}.
\end{proof}

\section{Examples}\label{sec:Examples}
In this section we discuss the computation of the Gromov width of  more general symplectic toric manifolds by providing some examples. The estimate given in Theorem~\ref{thm:Lu} seems to be a good candidate in the Fano case. We list some questions at the end of the section.

Let $(M, \omega)$ be a symplectic toric manifold. Recall the following expression of the moment polytope:
\begin{equation}
	P = \{ x \in \mathbb{R}^n \mid \langle x, \eta_i \rangle \leq \kappa_i \text{ for } i=1, \dots, N\}.
\end{equation}
We write $\eta_i$ as the $i$th column of $n \times N$ matrix in the following examples. The corresponding toric divisor is denoted by~$D_i$. We can vary the symplectic structure~$\omega$ by choosing different values of~$\kappa_i$. The choice of~$\kappa_i$ gives a symplectic form on~$M$ whenever no inequalities are redundant for defining~$P$.

The following example is from~\cite{Lu06}, which is obtained by resolving the singularity of the weighted projective space~$\mathbb{P}^4(1,1,2,2,2)$.
\begin{example}[{\cite[Example~4.1]{Lu06}}]\label{ex:Lu}
	Consider the symplectic toric manifold determined by the data
	\[
		[\eta_1 \cdots \eta_6] = 
		\begin{bmatrix}
			1 & 0 & 0 & 0 & -1 & 0\\
			0 & 1 & 0 & 0 & -2 & -1\\
			0 & 0 & 1 & 0 & -2 & -1\\
			0 & 0 & 0 & 1 & -2 & -1
		\end{bmatrix},
	\]
	$\kappa_2 = \kappa_5 = 1$ and $\kappa_i = 0$  for $i=1,3,4,6$. This is a generalized Bott manifold whose moment polytope is combinatorially equivalent to $\Delta^1 \times \Delta^3$.
	By Theorem~\ref{thm:main}, the Gromov width is equal to
	\[
		\kappa_2 + \kappa_3 + \kappa_4 + \kappa_6 =1.
	\]
	This coincides with the computation in~\cite{Lu06}.
\end{example}

When the symplectic toric manifold is Fano, we have the following theorem by Lu estimating the Gromov width from above. This estimate coincides with the value given in Theorem~\ref{thm:main}.

\begin{theorem}[\cite{Lu06} Theorem~1.2]\label{thm:Lu}
	Let $M$ be a Fano toric manifold equipped with the symplectic form~$\omega$ determined by the polytope~$P$. Suppose that $\sum_{i=1}^N a_i \eta_i = 0$ for some non-negative integers~$a_i$. Then
	\[
		w_G(M, \omega) \leq \sum_{i=1}^N a_i \kappa_i.
	\]
\end{theorem}
\begin{proof}
	The sum $\sum_{i=1}^N a_i \kappa_i$, which is the pairing between~$[\omega]$ and the homology class represented by the integers~$a_i$, does not change when we normalize the moment polytope~$P$. Since $S(\eta_i) = [D_i] \otimes q^{-1}t^{-\kappa_i}$ by Theorem~\ref{thm:MT} and $S$ is a group homomorphism, the proof follows as in Proposition~\ref{prop:upper_bound}.
\end{proof}

This bound may not be sharp if we drop the Fano condition, as we see in the following example.
\begin{example}\label{ex:toric_surface}
	Consider the symplectic toric manifold~$(M, \omega)$ determined by the following data
	\begin{align*}
	[\eta_1 \cdots \eta_9] = 
		&\begin{bmatrix}
			1 & 1 & 0 & -1 & -2 & -1 & 0 & 1 & 1\\
			0 & 1 & 1 & 1 & 1 & 0 & -1 & -2 & -1
		\end{bmatrix},\\
	[\kappa_1 \cdots \kappa_9] = 
		&\begin{bmatrix}
			6 & 7 & 6 & \phantom{-}6 & \phantom{-}7 & \phantom{-}6 & \phantom{-}6 & \phantom{-}7 & \phantom{-}6
		\end{bmatrix}.
	\end{align*}
	The moment polytope~$P$ is displayed in Figure~\ref{fig:toric_surface}. The upper bound given by Theorem~\ref{thm:Lu} is~$12$. On the other hand, the moment polytope~$P$ has area~$141 / 2$. Since 
\[
	\frac{w_G(M)^2}{2} \leq \mathrm{Vol}\; M = \mathrm{Area}\; P < 72,
\]
 the Gromov width should be less than~$12$.
	
	Indeed, we can show that the Gromov width is equal to~$21 / 2$ using the algorithm from~\cite[Section~6]{KK17}. See also \cite[Section~13.4]{MS17} and the references therein. Our underlying manifold~$M$ is diffeomorphic to $\mathbb{P}^2 \# 6\overline{\mathbb{P}^2}$, the space obtained by blowing up at $6$~points of~$\mathbb{P}^2$. Following the terminology of~\cite{KK17}, the vector $(18; 6,6,6,5,5,5)$ encodes the cohomology class~$[\omega]$ of the blowup form. Let $E$ denote the cohomology class of the exceptional divisor of the one-point blowup $\widetilde{M} \to M$. The problem of finding a symplectic embedding of the ball~$B^4(r)$ of radius~$r$ into~$M$ is reduced to finding a blowup form~$\widetilde{\omega}$ on~$\widetilde{M}$ such that $[\widetilde{\omega}] = [\omega] - \pi r^2 E$. By applying the standard Cremona move repeatedly, we can check that the vector $(18; 6,6,6,5,5,5, \pi r^2)$ transforms to a reduced vector with positive entries if and only if $\pi r^2 < 21/2$.
\end{example}
\begin{figure}
	\begin{tikzpicture}[scale = 0.2]
	\coordinate (v12) at (6,1);
	\coordinate (v23) at (1,6);
	\coordinate (v34) at (0,6);
	\coordinate (v45) at (-1,5);
	\coordinate (v56) at (-6,-5);
	\coordinate (v67) at (-6,-6);
	\coordinate (v78) at (-5,-6);
	\coordinate (v89) at (5,-1);
	\coordinate (v91) at (6,0);
	
	\draw[help lines] (-6,-6) grid (6,6);
	\draw[thick] (v12) -- (v23) -- (v34) -- (v45) -- (v56) -- (v67) -- (v78) -- (v89) -- (v91) -- cycle;
	
	\foreach \x in {12, 23, 34, 45, 56, 67, 78, 89, 91}
	\draw[fill=black] (v\x) circle (5.5pt);
	
	\end{tikzpicture}
	\caption{The moment polytope for Example~\ref{ex:toric_surface}.}
	\label{fig:toric_surface}
\end{figure}
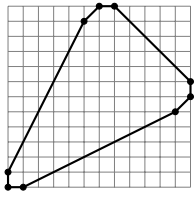

There are $5$ diffeomorphism types of Fano toric manifolds of dimension~$4$: $\mathbb{P}^1 \times \mathbb{P}^1$ and $\mathbb{P}^2 \# k\overline{\mathbb{P}^2}$ for $k = 0,1,2,3$. It is straightforward to check that the Gromov width of these manifolds are exactly the values given in Theorem~\ref{thm:Lu}. However, we do not know whether this still holds in higher dimensions.

\begin{example}\label{ex:Fano_bundle}
	Consider the toric manifold~$M$ with ray vectors given by
	\[
		[\eta_1 \cdots \eta_8] =
		\begin{bmatrix}
		1 & -1 & 0 & 0 & 0 & 0 & 0 & 0\\
		0 & 0 & 1 & 0 & -1 & -1 & 0 & 1\\
		1 & 0 & 0 & 1 & 1 & 0 & -1 & -1
		\end{bmatrix},
	\]
	and maximal cones determined by the polytope~$P$ when $\kappa_i = 1$ for all~$i$. Maximal cones are determined by the primitive collections\footnote{See \cite{Bat91} or \cite[Section~6.4]{CLS11} for the definition of primitive collections and primitive relations.}, explicitly given as
	\[
		\{1,2\}, \{3,5\}, \{3,6\}, \{3,7\}, \{4,6\}, \{4,7\}, \{4,8\}, \{5,7\}, \{5,8\}, \{6,8\}.
	\]
	This toric manifold is a fiber bundle over $\mathbb{P}^1$.	The fiber $\mathbb{P}^2 \# 3\overline{\mathbb{P}^2}$ is Fano but not a generalized Bott manifold. The total space itself is Fano as we can check by looking at primitive relations.
	
	First consider the case when the symplectic form~$\omega$ is given by setting ${\kappa_i = 1}$ for all~$i$.\footnote{Such symplectic manifold, that is, the one with $[\omega] = c_1(M)$, is called monotone.} Then since $\eta_3 + \eta_6 = 0$, by Theorem~\ref{thm:Lu} we have ${w_G(M, \omega) \leq 2}$. As we see in the first picture in Figure~\ref{fig:Fano_bundle}, we can embed a distorted cross-polytope~$\Diamond^3(2)$ into~$P$. By Proposition~\ref{prop:embedding} we conclude that
	\begin{equation}
		w_G(M, \omega) = 2.
	\end{equation}
		
	Now consider the case when the symplectic form~$\omega_{\epsilon}$ is given by setting $\kappa_1 = \epsilon$ with $0 < \epsilon < 1$ and $\kappa_i = 1$ for $i \neq 1$. Then the upper bound remains the same, while we cannot embed a distorted cross-polytope whose length is larger than $(5 + \epsilon) / 3$. To see this, let 
	\[
		\Diamond^3(\rho) := \mathrm{Conv} (L_1, L_2, L_3) \subset P
	\]
	be the convex hull of line segments $L_1, L_2, L_3$ of length~$\rho$ intersecting at some point~$p = (p_x, p_y, p_z)$. Let $\pi \colon \mathbb{R}^3 \to \mathbb{R}^2$ be the projection to the $yz$-plane. Note that $\pi$ does not decrease the length of line segments that are not parallel to the $x$-axis. 
	
	First consider the case when there exists at least one line segment, say~$L_1$, which is not orthogonal to either~$\eta_1$ or~$\eta_2$. Then $L_1$ has maximal length when it is parallel to the $x$-axis. By looking at the defining equations, we have
	\begin{equation}
		 \rho = \mathrm{length} \; L_1 \leq 1 + \epsilon - p_z.
	\end{equation}
	On the other hand,  as we see in the last picture of Figure~\ref{fig:Fano_bundle},
	\begin{equation}
		\rho 	\leq \min \{\mathrm{length} \; \pi(L_2), \mathrm{length} \; \pi(L_3) \} \leq 2 - |p_z| \leq 2 + p_z.
	\end{equation}
	Combining two inequalities 	we have $\rho \leq (3 + \epsilon) /2 < (5 + \epsilon) /3$.

	Now suppose that all $L_i$ are orthogonal to~$\eta_1$ or~$\eta_2$. Considering the case when each $L_i$ has maximal length, we obtain the following inequalities.
	\begin{align*}
		\begin{cases}
			\mathrm{length} \; L_1 \leq 2 + p_z + p_x \quad &\text{when } \mathrm{slope} \; L_1 = (1,0,-1),\\
			\mathrm{length} \; L_2 \leq 2 - p_z \quad &\text{when } \mathrm{slope} \; L_2 = (0,1,0),\\
			\mathrm{length} \; L_3 \leq 1 + \epsilon - p_x  \quad &\text{when } \mathrm{slope} \; L_3 = (0,0,1).
		\end{cases}
	\end{align*}
	Combining these inequalities we have $\rho \leq (5+ \epsilon ) / 3$. It follows that
	\begin{equation}
		\frac{5+ \epsilon}{3} \leq w_G(M, \omega_{\epsilon}) \leq 2.
	\end{equation}
\end{example}
\begin{figure}
	\begin{subfigure}[t]{0.5\textwidth}
		\centering
		\begin{tikzpicture}[scale = 1]
		\begin{scope}[color=gray!50, thin]
		\foreach \xi in {-1,0,1,2,3} { \draw (\xi, -1,2) -- (\xi, -1,-1) -- (\xi, 2, -1); }%
		\foreach \yi in {-1,0,1,2} {\draw (-1,\yi,2) -- (-1,\yi,-1) -- (3,\yi,-1);}%
		\foreach \zi in {-1,0,1,2} {\draw (-1,2,\zi) -- (-1,-1,\zi) -- (3,-1,\zi);}%
		\end{scope}
		
		\foreach \xi in {-1,0,1,2,3} {\draw (\xi,-1,2) -- (\xi, -1.1, 2) node[anchor=north] {\tiny{\xi}};}
		\draw (1,-1.4,2) node[anchor=north] {\small{$x$}};
		\foreach \yi in {-1,0,1,2} {\draw (-1,\yi,2) -- (-1.1,\yi, 2) node[anchor=east] {\tiny{\yi}};}
		\draw (-1.4,0.5,2) node[anchor=east] {\small{$y$}};
		\foreach \zi in {-1,0,1,2} {\draw (-1,2,\zi) -- (-1,2.1,\zi) node[anchor=south] {\tiny{\zi}};}
		\draw (-1, 2.4, 0.7) node[anchor=south] {\small{$z$}};
		
		\def\epsilon{1};
		\coordinate (v134) at (\epsilon - 1, 1, 1);
		\coordinate (v145) at (\epsilon - 1, 0, 1);
		\coordinate (v156) at (\epsilon, -1, 0);
		\coordinate (v167) at (\epsilon + 1, -1, -1);
		\coordinate (v178) at (\epsilon + 1, 0, -1);
		\coordinate (v183) at (\epsilon, 1, 0);
		\coordinate (v234) at (-1, 1, 1);
		\coordinate (v245) at (-1, 0, 1);
		\coordinate (v256) at (-1, -1, 0);
		\coordinate (v267) at (-1, -1, -1);
		\coordinate (v278) at (-1, 0, -1);
		\coordinate (v283) at (-1, 1, 0);
		
		\draw[color = blue] (-1,0,0) -- (1,0,0);
		\draw[color = blue] (-1,-1,0) -- (-1,1,0);
		\draw[color = blue] (-1,0,-1) -- (-1,0,1);
		
		\draw[thick, dashed] (v283)--(v278)--(v256);
		\draw[thick, dashed] (v278) -- (1,0,0);			
		\draw[thick, fill opacity = 0.6, fill=yellow]
		(v283) -- (v245) -- (1,0,0) -- cycle;
		\draw[thick, fill opacity = 0.6, fill=blue!50!green]
		(v245) -- (v256) -- (1,0,0) -- cycle;

		\draw (v283) -- (v234) -- (v245) -- (v256);
		\draw[dotted] (v283) -- (v278) -- (v267) -- (v256);
		\draw (v134) -- (v145) -- (v156) -- (v167) -- (v178) -- (v183) -- cycle;
		\draw (v134) -- (v234);
		\draw (v145) -- (v245);
		\draw (v156) -- (v256);
		\draw[dotted] (v167) -- (v267);
		\draw[dotted] (v178) -- (v278);
		\draw (v183) -- (v283);
		\draw[fill=black] (-1,0,0) circle (2pt);
		\end{tikzpicture}
		\caption{$\kappa_1 = 1$}
	\end{subfigure}
	~
	\begin{subfigure}[t]{0.5\textwidth}
		
		\centering
		\begin{tikzpicture}[scale = 1]
		\begin{scope}[color=gray!50, thin]
		\foreach \xi in {-1,0,1,2,3} { \draw (\xi, -1,2) -- (\xi, -1,-1) -- (\xi, 2, -1); }%
		\foreach \yi in {-1,0,1,2} {\draw (-1,\yi,2) -- (-1,\yi,-1) -- (3,\yi,-1);}%
		\foreach \zi in {-1,0,1,2} {\draw (-1,2,\zi) -- (-1,-1,\zi) -- (3,-1,\zi);}%
		\end{scope}
		
		\foreach \xi in {-1,0,1,2,3} {\draw (\xi,-1,2) -- (\xi, -1.1, 2) node[anchor=north] {\tiny{\xi}};}
		\draw (1,-1.4,2) node[anchor=north] {\small{$x$}};
		\foreach \yi in {-1,0,1,2} {\draw (-1,\yi,2) -- (-1.1,\yi, 2) node[anchor=east] {\tiny{\yi}};}
		\draw (-1.4,0.5,2) node[anchor=east] {\small{$y$}};
		\foreach \zi in {-1,0,1,2} {\draw (-1,2,\zi) -- (-1,2.1,\zi) node[anchor=south] {\tiny{\zi}};}
		\draw (-1, 2.4, 0.7) node[anchor=south] {\small{$z$}};
		
		\def\epsilon{0.6};
		\def\rho{(5 + \epsilon)/3};
		\def\px{(\epsilon - 1)*2/3};
		\def\pz{(1 - \epsilon)/3};
		\coordinate (v134) at (\epsilon - 1, 1, 1);
		\coordinate (v145) at (\epsilon - 1, 0, 1);
		\coordinate (v156) at (\epsilon, -1, 0);
		\coordinate (v167) at (\epsilon + 1, -1, -1);
		\coordinate (v178) at (\epsilon + 1, 0, -1);
		\coordinate (v183) at (\epsilon, 1, 0);
		\coordinate (v234) at (-1, 1, 1);
		\coordinate (v245) at (-1, 0, 1);
		\coordinate (v256) at (-1, -1, 0);
		\coordinate (v267) at (-1, -1, -1);
		\coordinate (v278) at (-1, 0, -1);
		\coordinate (v283) at (-1, 1, 0);
		\coordinate (p) at ({\px}, 0, {\pz});
		\coordinate (L11) at (-1,0, {(\epsilon +2)/3});
		\coordinate (L12) at ({(\epsilon + 2)/3}, 0, -1);
		\coordinate (L21) at ($(p) + (0, {-\rho +1}, 0)$);
		\coordinate (L22) at ($(p) + (0,1,0)$);
		\coordinate (L31) at ({\px}, 0, -1);
		\coordinate (L32) at ({\px}, 0, {-1 + \rho});
		
		\draw[color = blue] (L11) -- (L12);
		\draw[color = blue] (L21) -- (L22);
		\draw[color = blue] (L31) -- (L32);
		
		\draw[thick, dashed] (L12) -- (L31) -- (L22);
		\draw[thick, dashed] (L31) -- (L21);
		\draw[thick, fill opacity = 0.6, fill=yellow]
		(L12) -- (L22) -- (L32) -- cycle;
		\draw[thick, fill opacity = 0.6, fill=blue!50!green]
		(L12) -- (L21) -- (L32) -- cycle;
		\draw[thick, fill opacity = 0.5, fill=violet]
		(L11) -- (L21) -- (L32) -- cycle;
		\draw[fill opacity = 0.5, fill = red!50!orange]
		(L11) -- (L22) -- (L32) -- cycle;
		
		\draw (v283) -- (v234) -- (v245) -- (v256);
		\draw[dotted] (v283) -- (v278) -- (v267) -- (v256);
		\draw (v134) -- (v145) -- (v156) -- (v167) -- (v178) -- (v183) -- cycle;
		\draw (v134) -- (v234);
		\draw (v145) -- (v245);
		\draw (v156) -- (v256);
		\draw[dotted] (v167) -- (v267);
		\draw[dotted] (v178) -- (v278);
		\draw (v183) -- (v283);
		\draw[fill=black] (p) circle (2pt);
		
		\end{tikzpicture}
		\caption{$0 < \kappa_1 < 1$}
	\end{subfigure}
	
	\begin{subfigure}[t]{0.6\textwidth}
		\centering
		\begin{tikzpicture}[scale = 1]
		\begin{scope}[color=gray!50, thin]
		\draw (-1,1.7) -- (-1,-1) -- (1.7,-1);
		\end{scope}
		\foreach \yi in {-1,0,1} {\draw (\yi,-1) -- (\yi, -1.1)
			node [anchor = north] {\tiny{\yi}};}
		\draw (1.6, -1.1) node[anchor=north] {\small{$y$}};
		\foreach \zi in {-1,0,1} {\draw(-1,\zi) -- (-1.1,\zi) node[anchor=east] {\tiny{\zi}};}
		\draw (-1.1,1.6) node[anchor=east] {\small{$z$}};
		
		\def\delta{0.2}

		\coordinate (v34) at (1,1);
		\coordinate (v45) at (0,1);
		\coordinate (v56) at (-1,0);
		\coordinate (v67) at (-1,-1);
		\coordinate (v78) at (0,-1);
		\coordinate (v83) at (1,0);
		
		\draw (v34) -- (v45) -- (v56) -- (v67) -- (v78) -- (v83) -- cycle;
		\draw[color = blue] (0,-1) -- node[color = black, above right]{\small$\pi(L_2)$} (0,1);
		\draw[color = blue] (-1, -\delta) node[color = black, below right]{\small$\pi(L_3)$}-- (1 - \delta, -\delta);
		
		\draw[fill=blue] (0, -\delta) circle (2pt);
		\end{tikzpicture}
		\caption{The projection of $L_2, L_3$ to the $yz$-plane}
	\end{subfigure}
	\caption{Moment polytopes in Example~\ref{ex:Fano_bundle}. Distorted cross-polytopes are centered at block dots. The blue dot is the image of the center~$p$ under the projection to the $yz$-plane.}
	\label{fig:Fano_bundle}
\end{figure}
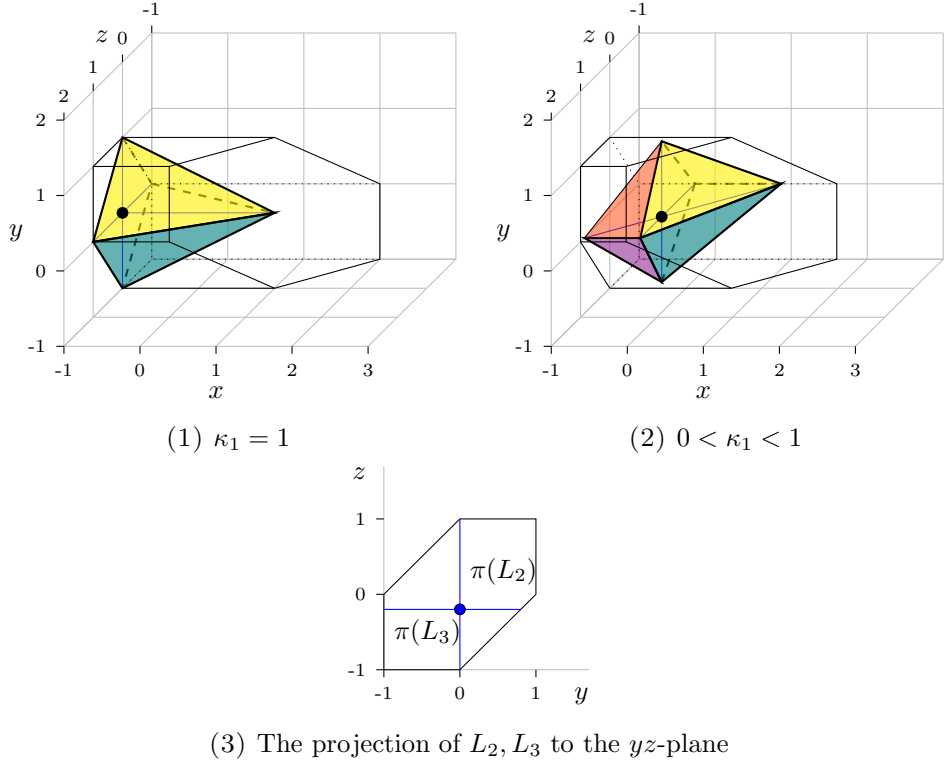

We list some questions which we could not solve.
\begin{Question}
	What is the Gromov width of $(M, \omega_{\epsilon})$ in~Example~\ref{ex:Fano_bundle}?
\end{Question}

\begin{Question}
	Consider monotone symplectic toric manifolds. Can the Gromov width be determined by combining Proposition~\ref{prop:embedding} and Theorem~\ref{thm:Lu}?
\end{Question}	

\begin{Question}
	Let $M$ be a closed symplectic toric manifold which is not necessarily Fano. If $\sum_{i=1}^N a_i \eta_i = 0$ for some non-negative integers~$a_i$, is it always true that	 $w_G(M) \leq \sum_{i=1}^N a_i \kappa_i?$
\end{Question}


\end{document}